\documentclass[a4paper,11pt,twoside]{amsart}

\usepackage[latin1]{inputenc}
\usepackage{amsmath}
\usepackage{amsthm}
\usepackage{amssymb}
\usepackage[all]{xy}
\usepackage{hyperref}

\newtheorem{thm}{Theorem}[section]

\newtheorem{prop}[thm]{Proposition}

\newtheorem{lem}[thm]{Lemma}
\newtheorem{cor}[thm]{Corollary}

\numberwithin{equation}{section}

\theoremstyle{definition}
\newtheorem{definition}[thm]{Definition}
\newtheorem{remark}[thm]{Remark}

\newcommand{\im}{\operatorname{im}}
\newcommand{\Db}{{\rm D}^{\rm b}}
\newcommand{\Dp}{{\rm D}_{\rm perf}}
\newcommand{\D}{{\rm D}}

\newcommand{\Aut}{{\rm Aut}}

\newcommand{\rk}{{\rm rk}}
\newcommand{\coh}{{\cat{Coh}}}
\newcommand{\qcoh}{{\cat{QCoh}}}

\newcommand{\Hom}{{\rm Hom}}
\newcommand{\Spec}{{\rm Spec}}
\newcommand{\Spf}{{\rm Spf}}

\newcommand{\iso}{\cong}

\newcommand{\id}{{\rm id}}

\newcommand{\dual}{\raisebox{-.5ex}{\makebox[.45em]{{\LARGE$\check{}$}}}}
\newcommand{\mono}{\hookrightarrow}
\newcommand{\epi}{\twoheadrightarrow}
\newcommand{\mor}[1][]{\xrightarrow{#1}}

\newcommand{\cat}[1]{\begin{bf}#1\end{bf}}
\newcommand{\Ext}{{\rm Ext}}

\newcommand{\KE}{{\rm Ker}}
\newcommand{\coker}{{\rm Coker}}
\renewcommand{\im}{{\rm Im}}
\renewcommand{\ker}{{\rm Ker}}
\newcommand{\COKE}{{\rm Coker}}
\newcommand{\Mod}[1]{{\ko_{#1}\text{-}\cat{Mod}}}
\newcommand{\cal}{\mathcal}
\newcommand{\ka}{{\cal A}}
\newcommand{\kb}{{\cal B}}

\newcommand{\ke}{{\cal E}}
\newcommand{\kf}{{\cal F}}
\newcommand{\kg}{{\cal G}}
\newcommand{\kh}{{\cal H}}

\newcommand{\km}{{\cal M}}
\newcommand{\ko}{{\cal O}}

\newcommand{\kx}{{\cal X}}

\newcommand{\NN}{\mathbb{N}}
\newcommand{\ZZ}{\mathbb{Z}}

\newcommand{\RR}{\mathbb{R}}
\newcommand{\CC}{\mathbb{C}}

\newcommand{\ddual}{\dual\dual}

\renewcommand{\to}{\xymatrix@1@=15pt{\ar[r]&}}
\newcommand{\lto}{\xymatrix@1@=15pt{&\ar[l]}}
\renewcommand{\leftarrow}{\xymatrix@1@=15pt{&\ar[l]}}
\renewcommand{\rightarrow}{\xymatrix@1@=15pt{\ar[r]&}}
\renewcommand{\mapsto}{\xymatrix@1@=15pt{\ar@{|->}[r]&}}
\renewcommand{\epi}{\xymatrix@1@=19pt{\ar@{->>}[r]&}}
\renewcommand{\twoheadrightarrow}{\xymatrix@1@=19pt{\ar@{->>}[r]&}}
\renewcommand{\hookrightarrow}{\xymatrix@W=2pt@1@=15pt{\ar@{^(->}[r]&}}
\newcommand{\congpf}{\xymatrix@1@=15pt{\ar[r]^-\sim&}}
\renewcommand{\cong}{\simeq}

\begin{document}
\title{Formal deformations and their categorical general fibre}

\author[D.\ Huybrechts, E.\ Macr\`i, and P.\ Stellari]{Daniel Huybrechts, Emanuele Macr\`i, and Paolo Stellari}

\address{D.H.: Mathematisches Institut,
Universit{\"a}t Bonn, Beringstr.\ 1, 53115 Bonn, Germany}
\email{huybrech@math.uni-bonn.de}

\address{E.M.: Department of Mathematics, University of Utah, 155 South 1400 East, Salt
Lake City, UT 84112-0090, USA} \email{macri@math.utah.edu}

\address{P.S.: Dipartimento di Matematica ``F. Enriques'',
Universit{\`a} degli Studi di Milano, Via Cesare Saldini 50, 20133
Milano, Italy} \email{paolo.stellari@unimi.it}

\begin{abstract}
We study the general fibre of a formal deformation over the formal
disk of a projective variety from the view point of abelian and
derived categories. The abelian category of coherent sheaves of
the general fibre is constructed directly from the formal
deformation and is shown to be linear over the field of Laurent
series. The various candidates for the derived category of the
general fibre are compared.

If the variety is a surface with trivial canonical bundle, we show
that the derived category of the general fibre is again a linear
triangulated category with a Serre functor given by the square of
the shift functor. The paper is a companion to \cite{HMS1}, where
the results are applied to Fourier--Mukai equivalences of K3
surfaces.
\end{abstract}

\keywords{Derived categories, deformations, $K$-trivial surfaces}

\subjclass[2000]{18E30, 14D15}

\maketitle



\section{Introduction}

Let $\pi:\kx\to\Spf(\CC[[t]])$  be a formal deformation of a
smooth complex projective variety $X$ given by an inductive system
of flat morphisms $\pi_n:\kx_n\to\Spec(\CC[t]/(t^{n+1}))$ with
$\kx_0=X$ and isomorphisms
$\kx_{n+1}\times_{\CC[t]/(t^{n+2})}\Spec(\CC[t]/(t^{n+1}))\cong\kx_n$
over $\CC[t]/(t^{n+1})$. Thus, $\kx$  as a ringed space is the
topological space $X$  with $\ko_\kx:=\lim\ko_{\kx_n}$ as its
structure sheaf. Interesting examples arise as formal
neighborhoods of an actual deformation of $X$ over a smooth
one-dimensional base, which may be algebraic or just a complex
disk.

In order to understand the generic behavior of certain classes of
varieties, it is often necessary to study the general fibre of
formal deformations of the type  $\pi:\kx\to\Spf(\CC[[t]])$. If
the deformation is given as the formal neighbourhood of an
algebraic deformation of $X$ over a curve, then the usual concept
of the scheme-theoretical general fibre yields a variety defined
over the function field of the curve. For arbitrary formal
deformations, e.g.\ obtained as formal neighborhoods of
deformations in non-algebraic directions, a geometric construction
of the general fibre as a rigid analytic variety is provided by
the work of Raynaud \cite{Ray}.

The aim of this paper is to present a categorical approach to the
general fibre. We construct the abelian category of coherent
sheaves on the general fibre directly without first passing to the
rigid analytic variety representing it geometrically. This
simplifies going back and forth from sheaves on the original
variety $X$ or its formal deformation $\kx$  to sheaves on the
general fibre. The passage from the abelian category to its derived category,
which plays a central role in the applications we have in mind, is
more difficult. Here we have to address subtle points related to
Verdier quotients of triangulated categories.

\medskip

To appreciate the results of this paper, we should briefly explain
the main application we developed in
\cite{HMS1}. For a smooth projective K3 surface $X$, Orlov proved
in \cite{Or1} that any autoequivalence of the bounded derived category of coherent sheaves $\Db(X)$ induces an
isomorphism of the total cohomology group $H^*(X,\ZZ)$ preserving
a natural weight-$2$ Hodge structure and the lattice structure
induced by the cup-product. In particular, there exists a
homomorphism of groups $\rho$ between the group of
autoequivalences $\Aut(\Db(X))$ and some orthogonal group (denoted
by $\mathrm{O}(\widetilde H(X,\ZZ))$) of the total cohomology
group of $X$ (see \cite{HMS1,Or1}).

Despite this nice result, a description of the image of $\rho$ had
been missing for some time. In \cite{Sz}, Szendr\H{o}i proposed a conjecture saying that $\rho$ should
send an equivalence to an isometry in $\mathrm{O}(\widetilde H(X,\ZZ))$
with the additional property that the orientation of some
$4$-dimensional positive definite subspace of $H^*(X,\RR)$ is
preserved.

In \cite{HMS1}, we gave a positive answer to this
conjecture using a deformation argument whose main steps are the
following. Given a smooth projective K3 surface $X$, we study a
very special formal deformation of $X$ based on its hyperk\"ahler
geometry. At this point, we argue that the derived category of the general fibre of such a
deformation, despite being $\CC((t))$-linear and not $\CC$-linear,
has the same basic features as the derived category of a generic
complex analytic K3 surface (i.e.\ a K3 surface with trivial
Picard group). The same conjecture has been solved in \cite{HMS}
for those surfaces. Hence, by the special choice of the
deformation, we can conclude that it holds true for $X$
as well.

The main result of this paper (Theorem \ref{thm:main} below) establishes some of the fundamental properties of the derived category of the general fibre which are needed in the above strategy.

\medskip

To state precisely this result, we first define the abelian and the derived category of the general
fibre of a formal deformation $\pi:\kx\to\Spf(\CC[[t]])$ of $X$.
Let $\coh(\kx)_0\subset \coh(\kx)$ be the full abelian subcategory
of coherent sheaves on $\kx$ which are torsion over $\CC[[t]]$,
i.e.\ the full subcategory of all sheaves $E\in\coh(\kx)$
supported on some $\kx_n$, for $n\gg0$. With this definition,
$\coh(\kx)_0$ is a Serre subcategory and  the quotient category
$$\coh(\kx_K):=\coh(\kx)/\coh(\kx)_0$$  is called the
\emph{abelian category of coherent sheaves on the general fibre}.
Here, $K$ denotes the quotient field of $\CC[[t]]$, i.e.\ the
field of all Laurent series. One can indeed  show that
$\coh(\kx_K)$ is a $K$-linear abelian category.

Next we denote by $\Db(\kx):=\Db_{\rm coh}(\Mod{\kx})$ the bounded derived category of the abelian category of $\ko_\kx$-modules with coherent cohomology.
This category has a full triangulated subcategory
$$\Db_0(\kx)\subset\Db(\kx)$$
consisting of all complexes with cohomology in $\coh(\kx)_0.$ The
Verdier quotient
$$
\Db(\kx_K):=\Db(\kx)/\Db_0(\kx)
$$
is the \emph{derived category of the general fibre} of the formal
deformation $\pi:\kx\to\Spf(\CC[[t]])$.

The fundamental properties of the derived category of the general fibre for smooth projective surfaces with trivial canonical bundle are explained in the following theorem which is the main result of the paper.

\begin{thm}\label{thm:main}
    Let $\kx\to\Spf(\CC[[t]])$ be a formal deformation of a smooth projective surface $X$
    with trivial canonical bundle. The derived category of the general fibre $\Db(\kx_K)$ is a
    $K$-linear triangulated category and the square of the shift functor defines a Serre functor.
    Moreover, there exists an exact $K$-linear equivalence
    $\Db(\kx_K)\cong\Db(\coh(\kx_K))$.
\end{thm}

The latter property, which is proved in Proposition \ref{prop:derallsame}, is extensively used in \cite{HMS1}. Both
interpretations of the derived category of the general fibre, as the Verdier quotient of triangulated categories and
as the bounded derived category of the abelian category of coherent sheaves on the general fibre, are used.
E.g.\ in \cite{HMS} all autoequivalences (and in particular spherical twists)
of  $\Db(\coh(X))$ are described for a generic (non-projective) K3 surface $X$. The
arguments apply as well to $\Db(\coh(\kx_K))$ for $\kx_K$ the general fibre of a very general formal
deformation of a projective K3 surface $X$. More precisely, we show that, up to shift, $\Db(\coh(\kx_K))$ contains
just one spherical object (namely the image of the structure sheaf $\ko_{\kx_K}$). On the other hand, in order to deform
a given autoequivalence of $\Db(X)$ to an autoequivalence of the derived category of the general fibre $\kx_K$ we need
to work with $\Db(\kx)$ and its quotient $\Db(\kx_K)$.
In the end we prove that the deformation of the autoequivalence of $\Db(X)$ to a
Fourier--Mukai equivalence $\Phi:\Db(\kx_K)\to\Db(\kx'_K)$ has kernel in the abelian category $\coh((\kx\times_R\kx')_K)$.

In the process of proving Theorem \ref{thm:main}, we will be considering a number of related technical results. To make the overview of the paper more complete, let us mention a few of them which will be particularly relevant in \cite{HMS1}:
\smallskip

(A) $\;$ The spaces of morphisms in quotient categories are often
difficult to describe. However, for the two natural quotients
$\coh(\kx)\to\coh(\kx_K)$ and $\Db(\kx)\to\Db(\kx_K)$ they are
simply given by the tensor pro\-duct with the quotient field $K$,
which makes both categories $K$-linear (Propositions
\ref{prop:Homcohgen} and \ref{pro:homgamma}).

(B) $\;$ One advantage of the categorical approach to the general
fibre is that the Fourier--Mukai machinery carries over easily.
For example, we prove that if the Fourier--Mukai kernel $\ke_0$ of an
equi\-valence $\Db(X)\congpf\Db(X')$ deforms to a complex $\ke$ on
the product of two formal deformations
$\kx,\kx'\to\Spf(\CC[[t]])$, then its restriction $\ke_K$ to the
general fibre defines again an equivalence
$\Db(\kx_K)\congpf\Db(\kx'_K)$ (Corollary
\ref{prop:genericalsoequiv}).

\medskip

We have not attempted to develop the theory in its most general
form. It would certainly be natural to study the general fibre of
formal deformations over more general formal rings from a
categorical perspective. Unfortunately, in that case, the results (e.g.\ the description of
the space of morphisms) would not nearly be as nice as in the
simple situation of deformations over $\CC[[t]]$. But even the
one-dimensional formal deformations studied here, should be useful
in other situations, although our discussion is tailored to the
application to Fourier--Mukai equivalences between K3 surfaces in
\cite{HMS1}.

\medskip

The plan of the paper is as follows. In Section \ref{sect:Coh} we define the abelian and derived categories of the general fibre of a formal deformation. We study their Hom-spaces and, in Sections \ref{sect:derfunc} and \ref{sect:FM}, we analyze the behavior of Fourier--Mukai transforms and Fourier--Mukai equivalences when passing to the derived categories of the general fibres.

In Section \ref{sect:Der} we complete the proof of Theorem \ref{thm:main}. As a first step, we compare the Hom-spaces and the Euler pairing on the general and special fibres of a formal deformation (Section \ref{subsect:Homs}). In Section \ref{subsec:Serre} we describe the Serre functor of the general fibre. Finally, in Section \ref{subsec:genfib}, we restrict to the case of smooth projective surfaces with trivial canonical bundle and prove the main theorem.

\bigskip

\noindent{\bf Notation.} Denote by $R:=\CC[[t]]$ the ring of power
series in $t$ which is a complete discrete valuation ring. Its
spectrum $\Spec(R)$ consists of two points: The closed point
$0:=(t)\in\Spec(R)$ with local ring $R$ and residue field $\CC$
and the generic point $(0)\in\Spec(R)$ with residue field
$K:=\CC((t))$, the field of Laurent series. Moreover, we put
$R_n:=\CC[t]/(t^{n+1})$ with the natural surjection
$R\twoheadrightarrow R_n$ defining a closed embedding
$\Spec(R_n)\subset\Spec(R)$, which is the $n$-th infinitesimal
neighbourhood of $0\in\Spec(R)$. The formal scheme $\Spf(R)$ is
then described by the increasing sequence of closed subschemes
$0=\Spec(R_0)\subset \Spec(R_1)\subset\ldots\subset
\Spec(R_n)\subset\ldots$. Throughout we will use the following
notations for the natural inclusions ($m<n$):
\begin{eqnarray*}&\iota_n : \kx_n \hookrightarrow
\kx~~~~{\rm and}~~~\iota:=\iota_0:X\hookrightarrow \kx;\\
&i_{m,n}: \kx_m \hookrightarrow \kx_n,~ i_n := i_{n,
n+1}:\kx_n\hookrightarrow\kx_{n+1}, ~~~{\rm
and}~~~j_n=i_{0,n}:X\hookrightarrow \kx_n.
\end{eqnarray*}


\section{The derived category of the general fibre}
\label{sect:Coh}

In this section we study the basic properties of the abelian
category $\coh(\kx_K)$ of coherent sheaves on the general fibre
and of the triangulated category $\Db(\kx_K)$. At the end of the
section we  also discuss the extension of the definition and of
some interesting basic properties of Fourier--Mukai functors in
the setting of formal deformations and of the derived categories
of their general fibres (Corollary \ref{prop:genericalsoequiv}, see also (B) in the introduction).

The reader not familiar with the notion of quotients of abelian categories by Serre subcategories or with that of Verdier quotiens is strongly encouraged to read the Appendix before proceeding with this section. For the convenience of the reader we list now the main abelian and triangulated categories which will be introduced in course of the paper. We also indicate the precise section where they are defined.

\begin{itemize}
	\item $\Mod\kx$: the abelian category of $\ko_\kx$-modules (Section \ref{sect:abgen});
	\item $\coh(\kx)$: the abelian category of coherent sheaves on $\kx$ (Section \ref{sect:abgen});
	\item $\coh(\kx)_0$: the Serre subcategory of $\coh(\kx)$ consisting of sheaves supported on some $\kx_n$ (Section \ref{sect:abgen});
	\item $\coh(\kx)_\mathrm{f}$: the full additive category $\coh(\kx)$ consisting of $\CC[[t]]$-flat sheaves (Section \ref{sect:abgen});
	\item $\coh(\kx_K)$: the quotient of the category $\coh(\kx)$ by $\coh(\kx)_0$ (Section \ref{sect:abgen});
\end{itemize}

\smallskip

\begin{itemize}
	\item $\Db(\Mod\kx)$: the bounded derived category of the abelian category $\Mod\kx$ (Section \ref{sect:dergen});
	\item $\Db(\coh(\kx))$: the bounded derived category of the abelian category $\coh(\kx)$ (Section \ref{sect:dergen})
	\item $\Db(\kx)=\Db_\mathrm{coh}(\Mod\kx)$: the full triangulated subcategory of $\Db(\Mod\kx)$ consisting of complexes with coherent cohomology (Section \ref{sect:dergen});
	\item $\Db_0(\kx)=\Db_{\coh(\kx)_0}(\Mod\kx)$: the full thick triangulated subcategory of $\Db(\kx)$ consisting of complexes with cohomology in $\coh(\kx)_0$ (Section \ref{sect:dergen});
	\item $\Db(\kx_K)$: the Verdier quotient $\Db(\kx)/\Db_0(\kx)$ (Section \ref{sect:dergen});
	\item $\Db_0(\coh(\kx))$: the full thick triangulated subcategory of $\Db(\coh(\kx))$ consisting of complexes with cohomology in $\coh(\kx)_0$ (Section \ref{sect:dergen});
	\item $\Dp(\kx_n)$: the full triangulated subcategory of perfect complexes on $\kx_n$ (Section \ref{sect:dergen});
	\item $\Db(\kx_K^c)$: the Verdier quotient $\Db(\coh(\kx))/\Db_0(\coh(\kx))$ (Section \ref{sect:dergen}).
\end{itemize}

\subsection{The abelian category of the general
fibre}\label{sect:abgen}

Given a formal deformation $\pi:\kx\to\Spf(R)$ of a smooth
projective variety $X$, the abelian category of all
$\ko_\kx$-sheaves will be denoted $\Mod{\kx}$. Any $\ko_\kx$-sheaf
$E$ yields an inverse system of $\ko_{\kx_n}$-sheaves
$E_n:=\iota^*_n E$ with $\ko_{\kx_n}$-linear transition maps
$E_n\to i_{m,n*}E_m$, for $n>m$, inducing isomorphisms
$i_{m,n}^*E_n\cong E_m$. Then $\lim E_n$ is again an
$\ko_\kx$-sheaf, but the natural homomorphism $E\to\lim E_n$ is in
general not an isomorphism. However, if we restrict to coherent
$\ko_X$-modules $E$, then indeed $E\cong\lim E_n$. This proves
that a coherent $\ko_\kx$-module is the same as an inverse system
of coherent $\ko_{\kx_n}$-sheaves $E_n$ together with transition
maps $E_n\to i_{m,n*}E_m$ inducing isomorphisms $i_{m,n}^*E_n\cong
E_m$ (see \cite[II.9]{HartAG} or \cite{IllFGA}).

By $\coh(\kx)\subset\Mod{\kx}$ we denote the full abelian
subcategory of all coherent sheaves on $\kx$ and we tacitly use
the equivalence of $\coh(\kx)$ with the abelian category of
coherent inverse systems as just explained. The restriction to
$\kx_n$ will be written as
$$\xymatrix{\coh(\kx)\ar[r]&\coh(\kx_n),}~~\xymatrix{E\ar@{|->}[r]&E_n.}$$
So in particular, $E_0\in\coh(X)$ will denote the restriction  of
a sheaf $E\in\coh(\kx)$ or $E_n\in\coh(\kx_n)$ to the special
fibre $X=\kx_0$. As we assume our formal scheme to be smooth, any
coherent sheaf on $\kx$ admits locally a finite free resolution.
However, since $\kx$ is not necessarily projective, locally free
resolutions might not exist globally.

\medskip

The category $\coh(\kx)$ of coherent sheaves on the formal
$R$-scheme $\kx$ is in a natural way an $R$-linear category. A
coherent sheaf $E\in\coh(\kx)$ has support on $\kx_n$ if
$t^{n+1}E=0$ and, as in the introduction, the subcategory
consisting of all sheaves having support on some $\kx_n$ is
denoted by $\coh(\kx)_0$.

A coherent sheaf $E\in\coh(\kx)$ is \emph{$R$-flat} if
multiplication with $t$ yields an injective homomorphism $t:E\to
E$. By $\coh(\kx)_{\rm f}\subset\coh(\kx)$ we denote the full
additive subcategory of $R$-flat sheaves. This subcategory is
clearly not abelian, but the two subcategories
$$\coh(\kx)_0,\coh(\kx)_{\rm f}\subset\coh(\kx)$$ define a torsion
theory for the abelian category $\coh(\kx)$. More precisely, there
are no non-trivial homomorphisms from objects in $\coh(\kx)_0$ to
objects in $\coh(\kx)_{\rm f}$ and every $E\in\coh(\kx)$ is in a
unique way an extension
$$\xymatrix{0\ar[r]& E_{\rm tor}\ar[r]& E\ar[r]& E_{\rm
f}\ar[r]& 0}$$ with $E_{\rm tor}\in\coh(\kx)_0$ and $E_{\rm
f}\in\coh(\kx)_{\rm f}$. Indeed, set $E_{\rm
tor}:=\bigcup\ker(t^n:E\to E)$, i.e.\ the $R$-torsion subsheaf of
$E$. The union must stabilize, as $E$ is coherent,  and $E_{\rm
f}:=E/E_{\rm tor}$ is $R$-flat. (Note that in general  this
torsion theory is not cotilting, i.e.\ not every $R$-torsion sheaf
is a quotient of an $R$-flat one.)

Let us now define the \emph{abelian category of coherent
sheaves on the general fibre} $$\coh(\kx_K):=\coh(\kx)/\coh(\kx)_0.$$

\begin{remark}
Since we divide out by a small subcategory, the quotient is a
category with Homsets. The same remark applies to all later
quotient constructions and we will henceforth  ignore the issue.
\end{remark}

The image of a sheaf $E\in\coh(\kx)$ under the natural projection
from $\coh(\kx)$ onto $\coh(\kx_K)$ is denoted $E_K$.

For two coherent sheaves $E,E'\in\coh(\kx)$ we
shall write $\Hom(E,E')$ for the group of homomorphisms in
$\coh(\kx)$ and $\Hom_K(E_K,{E'}_{\!\!\! K})$ for the group of
homomorphisms of their images $E_K,{E'}_{\!\!\! K}$ in
$\coh(\kx_K)$. The natural homomorphisms induced by the projection
will be denoted
$$\eta:\Hom(E,E')\to\Hom_K(E_K,{E'}_{\!\!\! K}).$$ By construction
of the quotient, any morphism $E_K\to {E'}_{\!\!\! K}$ in
$\coh(\kx_K)$ is an equivalence class of diagrams $(\xymatrix{E&
\ar[l]_-{s_0} E_0\ar[r]^-{g}& E'})$ with
$\ker(s_0),\COKE(s_0)\in\coh(\kx)_0$. The composition
$$(\xymatrix@C=12pt{E& \ar[l] E_0\ar[r]& E'})\circ
(\xymatrix@C=12pt{E'& \ar[l] E'_0\ar[r]& E''})$$ of two morphisms
$E_K\to {E'}_{\!\!\! K}$ and ${E'}_{\!\!\! K}\to{E''}_{\!\!\!\! K}
$ is naturally defined by means of the fibre product
$(\xymatrix@C=14pt{E& \ar[l]E_0\times_{E'}E'_0\ar[r]& E''})$.

Also note that $\coh(\kx)_{\rm f}\to\coh(\kx_K)$ is essentially
surjective, i.e.\ every object $F\in\coh(\kx_K)$ can be lifted to
an $R$-flat sheaf on $\kx$. Indeed, if $F=E_K$, then $(E_{\rm
f})_K\cong E_K=F$ and, therefore, $E_{\rm f}$ is an $R$-flat lift
of $F$.

\begin{remark}
As mentioned in the introduction, to the formal $R$-scheme $\kx$
one can associate the general fibre $\kx_K$ which is a rigid
analytic space (see \cite{Berthelot,Ray,RZ}). The abelian category
$\coh(\kx_K)$ is in fact equivalent to the category of coherent
sheaves on $\kx_K$, which explains the notation.
\end{remark}

\begin{prop}\label{prop:Homcohgen}
The abelian category $\coh(\kx_K)$ is $K$-linear and for all
$F,G\in\coh(\kx)$ the natural projection $\coh(\kx)\to\coh(\kx_K)$
induces a $K$-linear isomorphism
$$\Hom(F,G)\otimes_RK\congpf\Hom_{K}(F_K,G_K).$$
\end{prop}

\begin{proof}
As a quotient of the $R$-linear category $\coh(\kx)$, the category
$\coh(\kx_K)$ is also $R$-linear. The multiplication with $t^{-1}$
is defined as follows. Let $f \in \Hom_{K} (F_K, G_K)$ be a
morphism  re\-pre\-sented by $f:(\xymatrix{F& \ar[l]_{~s_0}
F_0\ar[r]^-{g}& G})$ with $\KE (s_0), \COKE (s_0) \in \coh
(\kx)_0$. Then set $t^{-1}f:(\xymatrix{F& \ar[l]_{~ts_0}
F_0\ar[r]^-{g}& G})$, which
 is a well-defined morphism in $\coh (\kx_K)$. This is because
the objects $\KE(ts_0)$ and $\COKE(ts_0)$ are in $\coh(\kx)_0$.
Moreover, one has $t (t^{-1} f) = f$ due to the following
commutative diagram
\begin{equation*}
\xymatrix@R=6pt@C=40pt{& F_0 \ar[dl]_{t s_0} \ar[dr]^{t g} \ar[dd]^{t \cdot \id} &\\
                 F && G.\\
                 & F_0 \ar[ul]^{s_0} \ar[ur]_{g} &
                }
\end{equation*}
The $K$-linearity of the composition is obvious.

Consider now the induced $K$-linear map
$$\xymatrix{\eta_K:\Hom (F, G) \otimes_R K \ar[r]&
 \Hom_{K} (F_K, G_K).}$$

To prove the injectivity of $\eta_K$, let $f \in \Hom (F, G)$ with
$\eta(f)=\eta_K(f) =0$. Then there exists a commutative diagram
\begin{equation*}
\xymatrix@R=10pt{& F' \ar[dl]_{s} \ar[dr]^{0} &\\
                 F \ar[rr]_{f} && G,
                }
\end{equation*}
with $\KE (s), \COKE (s) \in \coh (\kx)_0$ and hence $f$
factorizes through
$$\xymatrix{f:F \ar[r]^-{q}& \COKE (s) \ar[r]^-{f'}& G.}$$
Thus, if $t^n \COKE (s) =0$ for some $n>0$, then this yields $t^n
f  = f' \circ (t^n q) = 0$. In particular, $f\otimes 1 \in
\Hom(F,G)\otimes K$ is trivial.

In order to prove the surjectivity of $\eta_K$, we have to show
that for any $f \in \Hom_{K} (F_K, G_K)$ there exists an integer
$k$, such that $t^kf$ is induced by  a morphism $F\to G$ in
$\coh(\kx)$. Write $f:(\xymatrix{F&\ar[l]_{~s_0} F_0 \ar[r]^-{g}&
G})$ with $t^n \KE (s_0) = t^m \COKE (s_0) = 0$ for some positive
integers $m,n$. Consider the exact sequence
$$\xymatrix{0\ar[r]&\Hom (F', G) \ar[r]^-{\circ p}& \Hom (F_0, G) \ar[r]^-{\circ i}&
\Hom (\KE (s_0), G)}$$ induced by the natural projection
$p:F_0\twoheadrightarrow F':=\im(s_0)$ and its kernel
$i:\ker(s_0)\hookrightarrow F_0$. Since $ (t^n g) \circ i = g
\circ (t^n i) = 0$, there exists a (unique) homomorphism $g' : F'
\to G$ such that $g' \circ p = t^n g$. This yields the commutative
diagram
\begin{equation*}
\xymatrix@R=6pt@C=40pt{& F_0 \ar[dl]_{s_0} \ar[dr]^{t^n g} \ar[dd]^{p} &\\
                 F && G,\\
                 & F' \ar@{_{(}->}[ul] \ar[ur]_{g'} &
                }
\end{equation*}
which allows one to represent $t^nf$ by $(\xymatrix{F&
\ar@{_{(}->}[l] F'\ar[r]^-{g'}& G})$.

As $F/F'\cong\COKE(s_0)$ is annihilated by $t^m$, the homomorphism
$t^mg':F'\to G$ lifts to a homomorphism $g'':F\to G$, i.e.\
$g''|_{F'}=t^mg'$. This yields the commutative diagram
\begin{equation*}
\xymatrix@R=6pt@C=40pt@W=20pt{& ~F' \ar@{_{(}->}[dl] \ar[dr]^{t^m g'} \ar@{^{(}->}[dd] &\\
                 F && G.\\
                 & F \ar[ul]^{\id} \ar[ur]_{g''} &
                }
\end{equation*}
Hence $t^{m+n}f$ is represented by $(\xymatrix{F& \ar[l]_{~\id}
F\ar[r]^-{g''}& G})$, i.e.\ $t^{m+n}f=\eta(g'')$.
\end{proof}

\subsection{The derived category of the general fibre}\label{sect:dergen}
Let $\pi:\kx\to \Spf(R)$ be a formal deformation of $X$ and
consider the bounded derived category of $\kx$ defined as $$\Db
(\kx):=\Db_{\rm coh}(\Mod{\kx}),$$ which by definition is an
$R$-linear triangulated category.

\begin{remark}\label{rem:notperfect}
We will always tacitly use the well-known (at least for schemes)
fact that any bounded complex with coherent cohomology on a smooth
formal scheme is \emph{perfect}, i.e.\ locally isomorphic to a
finite complex of locally free sheaves of finite type (see e.g.\
\cite[Cor.\ 5.9]{Ill}). In other words $\Dp(\kx)\cong\Db(\kx)$.
This is however not true for $\kx_n$, $n>0$. Indeed, e.g.\ for
$n=1$ one has ${\rm Tor}_{R_1}^i(R_0,R_0)\cong R_0$ for all
$i\geq0$. So, the $R_1$-module $R_0$ does not admit a finite free
resolution. So we will have to work with $${\rm D}_{\rm
perf}(\kx_n)\subset\Db(\kx_n),$$ the full triangulated subcategory
of perfect complexes on $\kx_n$.
\end{remark}

Recall that for the noetherian scheme $\kx_n$ the functor
$$\Db(\coh(\kx_n))\congpf\Db(\kx_n):=\Db_{\rm coh}(\Mod{\kx_n})$$ is
an equivalence. Contrary to the case of a noetherian scheme, the
natural functor
\begin{equation}\label{eqn:natfunct}
\xymatrix{\Db(\coh(\kx))\ar[r]&\Db(\kx)=\Db_{\rm coh}(\Mod\kx)}
\end{equation}
is in general not an equivalence. However, \eqref{eqn:natfunct}
induces an equivalence between the full subcategories of
$R$-torsion complexes. To be more precise, let
$$\Db_0(\kx):=\Db_{\coh(\kx)_0}(\Mod\kx)\subset \Db(\kx)~~~{\rm
~~and~~}~~~\Db_0(\coh(\kx))\subset\Db(\coh(\kx))$$ be the full
triangulated subcategories of complexes with cohomology contained
in $\coh(\kx)_0$. Then one has:

\begin{prop}\label{prop_dertorsion}{\rm  i)} The natural functor
$\Db(\coh(\kx)_0)\to\Db(\coh(\kx))$ induces an equivalence
$$\xymatrix{\Db(\coh(\kx)_0)\ar[r]^-\sim&\Db_0(\coh(\kx)).}$$
{\rm ii)} The natural functor $\Db(\coh(\kx)_0)\to\Db(\kx)$
induces an equivalence
$$\xymatrix{\Db(\coh(\kx)_0)\ar[r]^-\sim&\Db_0(\kx).}$$
\end{prop}

\begin{proof}
i) It suffices to show (see the dual version of \cite[Lemma
3.6]{HFM}) that for any monomorphism $f : E \hookrightarrow E'$ in
$\coh (\kx)$ , with $E \in \coh (\kx)_0$, there exists $g : E'
\mor E''$, with $E''\in \coh (\kx)_0$ such that $g \circ f$  is
injective.

By the Artin--Rees Lemma, we know that the filtration $E_k := E
\cap t^k E'$ is $t$-stable, that is, there is some $n \in \NN$
such that $tE_k = E_{k+1}$, whenever $k \geq n$. Let $\ell$ be a
positive integer such that $t^\ell E = 0$ and let $g:E'\to E'' :=
E' / t^{n+\ell}E'$ be the projection. The composition $g\circ f$
is injective, as $\ker(g\circ f)= E_{n+\ell} = t^\ell E_n
\hookrightarrow t^\ell E = 0$.

ii) We follow Yekutieli \cite{Yek}, but see also \cite{AJL}. Let
$\qcoh(\kx)\subset\Mod\kx$ be the full abelian subcategory of
quasi-coherent sheaves on $\kx$, i.e.\ of sheaves which are
locally cokernels of $\ko_\kx^I\to\ko_\kx^J$ for some index sets
$I,J$.  Then define $\qcoh(\kx)_{\rm d}\subset\qcoh(\kx)$ as the
full thick abelian subcategory of discrete quasi-coherent sheaves (see the Appendix for the definition of thick abelian subcategory).
By definition,  a sheaf $E$ on $\kx$ is  \emph{discrete} if the
natural functor $\Gamma_{\rm d}(E):=\lim\kh om(\ko_{\kx_n},E)\to
E$ is an isomorphism.

Clearly, a coherent sheaf on $\kx$ is discrete if and only if it
is $R$-torsion, i.e.\ $\coh(\kx)_0=\coh(\kx)\cap\qcoh(\kx)_{\rm
d}$ which is a thick subcategory of $\qcoh(\kx)_{\rm d}$.
Moreover, by \cite[Prop.\ 3.8]{Yek} every $E\in\qcoh(\kx)_{\rm d}$
is the limit of coherent $R$-torsion sheaves. Thus $\Db_{\coh(\kx)_0}(\qcoh(\kx)_{\rm d})$ is the same as $\Db_\mathrm{coh}(\qcoh(\kx)_{\rm d})$. Lemma
\ref{lem:accessory4} below gives an equivalence $\Db(\coh(\kx)_0)\cong\Db_{\coh(\kx)_0}(\qcoh(\kx)_\mathrm{d})$. Hence we conclude the equivalence
$\Db(\coh(\kx)_0)\cong\Db_{\rm coh}(\qcoh(\kx)_{\rm d})$.

Finally, one applies \cite[Thm.\ 4.8]{Yek} which asserts that
the natural functor induces an equivalence of $\Db(\qcoh(\kx)_{\rm
d})$ with the full triangulated subcategory of $\Db(\Mod\kx)$ of
all complexes with cohomology in $\qcoh(\kx)_{\rm d}$. (The
inverse functor is given by $R\Gamma_{\rm d}$.) Adding the
condition that the cohomology be coherent proves ii).
\end{proof}

\begin{lem}\label{lem:accessory4}
Let $\ka \subseteq \kb$ be a full thick abelian subcategory of an
abelian category $\kb$ with infinite direct sums. Assume that
every object of $\kb$ is the direct limit of its subobjects
belonging to $\ka$ and that  $\ka$ is noetherian (i.e.\ every
ascending sequence of subobjects is stationary). Then the natural
functor yields an equivalence
$$\Db (\ka) \congpf \Db_{\ka} (\kb),$$
where $\Db_{\ka} (\kb)$ is the full triangulated subcategory of $\Db(\kb)$ of complexes with cohomology in $\ka$.
\end{lem}

\begin{proof}
Let $f : E \epi E'$ be a surjection in $\kb$, with $E' \in \ka$.
We need to show that there exists a morphism $g: G\to E$ with $G
\in \ka$ such that $f \circ g : G \to E'$ is again surjective (see
e.g.\ \cite[Lemma 3.6]{HFM}).

By assumption, there exists a direct system of objects $\{ E_i \}$
in $\ka$ such that $\lim E_i \cong E$. Hence, there exists a
surjection $j: \bigoplus_i E_i \epi E \epi E'$.

Then let $E_k':=\im\left(\overset{k}{\underset{i=0}{\bigoplus}}
E_i \to E \epi E'\right)$, which form an ascending sequence of
subobjects of $E'$. Since $\ka$ is a noetherian, the sequence $\{
E_k' \}$ stabilizes, and,  as $j$ is surjective, $E_k' = E'$ for
$k\gg0$. Then set
$G:=\overset{k}{\underset{i=0}{\bigoplus}}E_i\in\ka$, for some
$k\gg0$, and let $g$ be the natural morphism.
\end{proof}

\begin{remark}\label{rem:smallest}
i) The equivalences of Proposition \ref{prop_dertorsion} put in
one diagram read
\begin{equation}
\Db(\coh(\kx)_0)\cong\Db_0(\coh(\kx))\cong\Db_0(\kx).
\end{equation}

ii) The categories $\Db_0(\kx)\subset\Db(\kx)$ and
$\Db(\coh(\kx)_0)\subset\Db(\coh(\kx))$ can also be described as
the smallest full triangulated subcategories containing all
$R$-torsion coherent sheaves. Here, a sheaf $E\in\coh(\kx_n)$ is
at the same time considered as an object in $\Db(\kx)$ and
$\Db(\coh(\kx))$. This is clear, as any bounded complex with
$R$-torsion cohomology can be filtered (in the triangulated sense)
with quotients being translates of such sheaves.
\end{remark}

\medskip

In the introduction we have already defined the \emph{derived category of the general fibre}
$\Db(\kx_K)$, i.e.\ the Verdier quotient
$$\Db(\kx_K):=\Db(\kx)/\Db_0(\kx)=\Db_{\rm coh}(\Mod\kx)/\Db_{\coh(\kx)_0}(\Mod\kx).$$ One can also consider the quotient
$\Db(\coh(\kx))/\Db_0(\coh(\kx))$ which, for a lack of a better
notation, will be called
$$\Db(\kx_K^c):=\Db(\coh(\kx))/\Db_0(\coh(\kx)).$$
(For a thorough discussion of the Verdier quotient see the Appendix.)

In both cases, the quotients are triangulated and the natural
projections
\begin{equation}\label{disp:twoproj}
\xymatrix{\Db(\kx)\ar[r]&\Db(\kx_K)~~~{\rm and}~~~
\Db(\coh(\kx))\ar[r]&\Db(\kx_K^c)}\end{equation} are exact. The
image of a complex $E$ under any of these projections shall be
denoted $E_K$.

\begin{remark}
As $\coh(\kx)_0\subset\coh(\kx)\subset \Mod\kx$ are Serre
subcategories, the subcategories $\Db_0(\Mod\kx)\subset\Db(\kx)$
and $\Db_0(\coh(\kx))\subset \Db(\coh(\kx))$ are thick. This means that the
direct summands of their objects are again contained in the
subcategories. This has the consequence that the kernel of the two
projections in (\ref{disp:twoproj}) are indeed $\Db_0(\kx)$ and $\Db_0(\coh(\kx))$ respectively.
\end{remark}

\begin{prop}\label{pro:homgamma}
The triangulated category $\Db(\kx_K)$ is $K$-linear and for all
$E,E'\in\Db(\kx)$ the natural projection $\Db(\kx)\to\Db(\kx_K)$
induces $K$-linear isomorphisms
$$\Hom_{\Db(\kx)}(E,E')\otimes_RK\congpf\Hom_{\Db(\kx_K)}(E_K,E'_K).$$
Similarly, $\Db(\kx_K^c)$ is $K$-linear and for
$E,E'\in\Db(\coh(\kx))$ one has
$$\Hom_{\Db(\coh(\kx))}(E,E')\otimes_RK\congpf\Hom_{\Db(\kx^c_K)}(E_K,E'_K).$$
In particular, $\Db(\kx_K)$ and $\Db(\kx_K^c)$ have
finite-dimensional Hom-spaces over $K$.
\end{prop}

\begin{proof}
As we work with bounded complexes, the proof of Proposition
\ref{prop:Homcohgen} carries over.
\end{proof}




\subsection{Derived functors and Fourier--Mukai transforms}\label{sect:derfunc}

First of all we prove that the usual derived functors (tensor
product, pull-back push-forward, Hom's) are well-defined in the
geometric setting we are dealing with.

\begin{prop}\label{prop:recallderfunc}
Let $f,g:\kx\to\kx'$ be morphisms of smooth and proper formal
schemes over $\Spf(R)$ and assume $f$ to be proper. Then the
following $R$-linear functors are defined:
\[
R\kh
om_{\kx}(-,-):\Db(\kx)^{\mathrm{op}}\times\Db(\kx){\xymatrix@1@=19pt{\ar[r]&}}\Db(\kx),
\]
\[
(-)\otimes^L(-):\Db(\kx)\times\Db(\kx){\xymatrix@1@=19pt{\ar[r]&}}\Db(\kx),
\]
\[
Lg^*:\Db(\kx'){\xymatrix@1@=19pt{\ar[r]&}}\Db(\kx),
\]
\[
Rf_*:\Db(\kx){\xymatrix@1@=19pt{\ar[r]&}}\Db(\kx'),
\]
\[
R\Hom_{\Db(\kx)}(-,-):\Db(\kx)^{\mathrm{op}}\times\Db(\kx){\xymatrix@1@=19pt{\ar[r]&}}\Db(R\text{-}\mathbf{mod}),
\]
where we denote by $R\text{-}\mathbf{mod}$ the abelian category of
$R$-modules of finite rank.
\end{prop}

\begin{proof}
Due to \cite{Spalt}, the functors previously considered are all
well-defined if we work with unbounded derived categories of
modules $\D (\Mod{\kx})$, $\D (\Mod{\kx'})$ and $\D
(R\text{-}\mathbf{Mod})$ (here $R\text{-}\mathbf{Mod}$ denotes the
abelian category of $R$-modules). To prove the proposition, we
only have to show that, by restricting the domain to the
corresponding derived categories of bounded complexes with
coherent cohomology, the images of these functors are still the
bounded derived categories of complexes with coherent cohomology.
This is clear since all complexes are perfect.\end{proof}

All the basic properties of the functors considered in the
previous proposition (e.g.\ commutativity, flat base change,
projection formula) hold in the formal context. For an object
$E\in\Db(\kx)$ a \emph{trace map} $\mathrm{tr}_{E}:E\dual\otimes
E\to\ko_{\kx}$ is well-defined (see \cite{Ill}).

Passing to the triangulated category $\Db(\kx_K)$ of the generic
fibre, the result in Proposition \ref{prop:recallderfunc} still
hold. Indeed, all the functors are $R$-linear and hence they
factorize through the category $\Db(\kx_K)$. Indeed,
$\kf\in\Db(\kx)$ is contained in $\Db_0(\kx)$ if and only if
$t^n\kf=0$ for $n\gg0$. Since the functors are $R$-linear, the
same would hold for the image of $\kf$ which would therefore  as
well be contained in the subcategory $\Db_0$. Thus we get the
following list of functors:
\[
R\kh
om_{\kx_K}(-,-):\Db(\kx_K)^{\mathrm{op}}\times\Db(\kx_K){\xymatrix@1@=19pt{\ar[r]&}}\Db(\kx_K),
\]
\[
(-)\otimes^L(-):\Db(\kx_K)\times\Db(\kx_K){\xymatrix@1@=19pt{\ar[r]&}}\Db(\kx_K),
\]
\[
Lg^*:\Db(\kx'_K){\xymatrix@1@=19pt{\ar[r]&}}\Db(\kx_K),
\]
\[
Rf_*:\Db(\kx_K){\xymatrix@1@=19pt{\ar[r]&}}\Db(\kx'_K),
\]
\[
R\Hom_{\Db(\kx_K)}(-,-):\Db(\kx_K)^{\mathrm{op}}\times\Db(\kx_K){\xymatrix@1@=19pt{\ar[r]&}}\Db(K\text{-}\mathbf{vect}),
\]
where we denote by $K\text{-}\mathbf{vect}$ the abelian category
of finite dimensional $K$-vector spaces. Of course, all the usual
relations between these functors continue to hold in $\Db(\kx_K)$.
In particular, given an object $E_K\in\Db(\kx_K)$, its \emph{dual}
$E_K\!\!\!\!\dual\in\Db(\kx_K)$  is well-defined and
$E_K\!\!\!\!\ddual\iso E_K$. Moreover, we have a \emph{trace map}
$\mathrm{tr}_{E_K}:E_K\!\!\!\!\dual\otimes E_K\to\ko_{\kx_K}$,
where $\ko_{\kx_K}$ is the image of $\ko_\kx$ in $\Db(\kx_K)$.



\medskip

Using those facts, we define Fourier--Mukai functors for formal
deformations or for the derived categories of the general fibres.
Indeed, consider two smooth and proper formal schemes
$\kx\to\Spf(R)$ and $\kx'\to\Spf(R)$ of dimension $d$ respectively
$d'$, with special fibres $X$ respectively $X'$. The fibre product
$\kx\times_R\kx'\to \Spf(R)$, described by the inductive system
$\kx_n\times_{R_n}\kx_n'$,  is again smooth and proper and its
special fibre is $X\times X'$. The two projections shall be called
$q:\kx\times_R\kx'\to\kx$ and $p:\kx\times_R\kx'\to\kx'$.

Let $\ke\in\Db(\kx\times_R\kx')$. Due to the results in the
previous section, one can consider the induced
\emph{Fourier--Mukai transform}
$$\xymatrix{\Phi_\ke:\Db(\kx)\ar[r]&\Db(\kx')},~~~\xymatrix{E\ar@{|->}[r]&Rp_*(q^*E\otimes^L\ke).}$$
As before,  $\Phi_\ke$ is $R$-linear, for $\ke$ lives on the fibre
product over $\Spf(R)$.

For two given Fourier--Mukai transforms
$\Phi_{\ke}:\Db(\kx)\to\Db(\kx')$ and
$\Phi_\kf:\Db(\kx')\to\Db(\kx'')$ with $\kx''$ smooth and proper
formal scheme over $\Spf(R)$, the composition
$\Phi_\kf\circ\Phi_\ke$ is again a Fourier--Mukai transform with
kernel $\kf\ast\ke:=(p_{\kx,\kx''})_*(\ke\boxtimes\kf)$, where
$p_{\kx,\kx''}:\kx\times\kx'\times\kx''\to\kx\times\kx''$ is the
natural projection.

Left and right adjoint functors of a Fourier--Mukai transform
$\Phi_\ke$ can be constructed as Fourier--Mukai transforms as
follows. The left adjoint $\Phi_{\ke_{\rm L}}$ and the right
adjoint $\Phi_{\ke_{\rm R}}$ are the Fourier--Mukai transforms
with kernel
$$\ke_{\rm L}:=\ke\dual\otimes p^*\omega_{p}[d']~~~
{\rm respectively~~~}\ke_{\rm R}:=\ke\dual\otimes
q^*\omega_{q}[d],$$ where $d=\dim(\kx)$ and $d'=\dim(\kx')$. (To
adapt to this context the standard proof that those kernels define
the left and right adjoints of $\Phi_\ke$, we actually need Lemma
\ref{cor:AJL} which will be proved later.)

The adjunction morphisms $\Phi_{\ke_{\rm L}}\circ\Phi_\ke\to{\rm
id}_{\Db(\kx)}$ and $\Phi_{\ke}\circ\Phi_{\ke_{\rm R}}\to{\rm
id}_{\Db(\kx')}$ are isomorphisms if and only if the natural
morphisms ${\rm tr}_\kx:\ke_{\rm L}\ast\ke\to\ko_{\Delta_{\kx}}$
respectively ${\rm tr}_{\kx'}:\ke\ast\ke_{\rm
R}\to\ko_{\Delta_{\kx}}$ induced by the trace morphisms are
isomorphisms. Here $\Delta_\kx$ and $\Delta_{\kx'}$ denote the
relative diagonals in $\kx\times_R\kx$ respectively
$\kx'\times_R\kx'$. (Sometimes (see e.g.\ \cite{Cal1}) the
construction of the adjunction morphisms uses
Grothendieck--Verdier duality for certain embeddings, e.g.\ for
$\kx\times_R\kx'\hookrightarrow\kx\times_R\kx'\times_R\kx$. This
can easily be replaced by an argument using relative duality over
$R$ in the sense of Lemma \ref{cor:AJL} for the two sides.
We leave the details to the reader.)

\begin{remark}
Everything said also works for the non-reduced schemes
$\kx_n\to\Spec(R_n)$ and $\kx_n'\to\Spec(R_n)$ with the only
difference that we have to assume now that the Fourier--Mukai
kernel $\ke_n\in\Db(\kx_n\times_{R_n}\kx_n')$ is perfect. Then one
can consider the two $R_n$-linear functors
$$\xymatrix{\Phi_{\ke_n}:\Db(\kx_n)\ar[r]&\Db(\kx_n')}~~~{\rm
and}~~~\xymatrix{\Phi_{\ke_n}:\Dp(\kx_n)\ar[r]&\Dp(\kx_n').}$$
\end{remark}

Analogously, one wants to define the Fourier--Mukai transform
$$\xymatrix{\Phi_\kf:\Db(\kx_K)\ar[r]&\Db(\kx'_K)}$$
associated to an object $\kf\in\Db((\kx\times_R\kx')_K)$.

As the objects of $\Db(\kx\times_R\kx')$ are the same as those of $\Db((\kx\times_R\kx')_K)$ (see the Appendix), take
$\ke\in\Db(\kx\times_R\kx')$ such that $\ke_K\cong\kf$. Then, by
the $R$-linearity, the Fourier--Mukai transform
$\Phi_\ke:\Db(\kx)\to\Db(\kx')$ descends to a Fourier--Mukai
transform $\Phi_\kf:\Db(\kx_K)\to\Db(\kx'_K)$, i.e.\ one has a
commutative
diagram $$\xymatrix{\Db(\kx)\ar[d]_Q\ar[r]^{\Phi_\ke}&\Db(\kx')\ar[d]^Q\\
\Db(\kx_K)\ar[r]_{\Phi_\kf}&\Db(\kx_K').}$$
As the objects of $\Db(\kx)$ and $\Db(\kx_K)$ coincide, it is enough to check that $\Phi_\ke(\Db_0(\kx))=\Db_0(\kx')$.
Indeed,
$\kg\in\Db(\kx)$ is contained in $\Db_0(\kx)$ if and only if
$t^n\kg=0$ for $n\gg0$. As $\Phi_\ke$ is $R$-linear, this would
imply $t^n\Phi_{\ke}(\kg)=0$ and hence
$\Phi_{\ke}(\kg)\in\Db_0(\kx')$. The behavior of $\Phi_\kf$ on the level of morphisms is determined by Proposition \ref{pro:homgamma}. Moreover, $\Phi_\kf$ is
independent of the chosen lift $\ke$. From this, it follows that
right and left adjoints of Fourier--Mukai transforms as well as
trace maps pass to the triangulated category of the generic fibre.




\subsection{Fourier--Mukai equivalences of the general fibre}\label{sect:FM}
Here we will show that if the kernel of a Fourier--Mukai
equivalence deforms to a complex on some finite order deformation
or even to the general fibre, then it still induces derived
equivalences of the finite order deformations or general fibres,
respectively. This is certainly expected, as `being an
equivalence' should be an open property and indeed the proof
follows the standard arguments.

Consider two smooth and proper formal schemes  $\kx\to\Spf(R)$ and
$\kx'\to\Spf(R)$, with special fibres $X$ respectively $X'$. The
fibre product $\kx\times_R\kx'\to \Spf(R)$, described by the
inductive system $\kx_n\times_{R_n}\kx_n'\to\Spec(R_n)$,  is again
smooth and proper and its special fibre is $X\times X'$.

Set $\kx_\infty:=\kx$, $\kx'_\infty:=\kx'$, and $R_\infty:=R$
(notice that $\Dp(\kx\times_R\kx')\iso\Db(\kx\times_R\kx')$).

\begin{prop}\label{prop:extequiv}
Let $\ke_n\in\Dp(\kx_n\times_{R_n}\kx'_n)$, with
$n\in\NN\cup\{\infty\}$, be such that its restriction
$\ke_0:=Lj_n^*\ke_n$ to the special fibre $X\times X'$ is the
kernel of a Fourier--Mukai equivalence
$\Phi_{\ke_0}:\Db(X)\congpf\Db(X')$. Then the Fourier--Mukai
transform $\Phi_{\ke_n}:\Dp(\kx_n)\to\Dp(\kx'_n)$ is an
equivalence.
\end{prop}

\begin{proof}
It suffices to show that in both cases left and right adjoint
functors are quasi-inverse. Complete the trace morphism to a
distinguished triangle
$$\xymatrix{(\ke_n)_{\rm L}\ast\ke_n\ar[r]^-{{\rm
tr}_{\kx_n}}&\ko_{\Delta_{\kx_n}}\ar[r]&\kg_n.}$$ Restricting it
to the special fibre yields the distinguished triangle
$$\xymatrix{(\ke_0)_{\rm L}\ast\ke_0\ar[r]^-{{\rm
tr}_X}&\ko_{\Delta_{X}}\ar[r]&\kg_0.}$$ (Use that the pull-back of
the trace is the trace. Also the restriction of $(\ke_n)_{\rm L}$
yields the kernel of the left adjoint of the restriction $\ke_0$.)

As by assumption $\Phi_{\ke_0}:\Db(X)\to\Db(X')$ defines an
equivalence, the cone $\kg_0$ is trivial. Thus,
$\kg_n\in\Db(\kx_n\times_{R_n}\kx'_n)$ has trivial restriction to
the special fibre $X\times X'$ and, therefore, $\kg_n\cong0$. This
shows that ${\rm tr}_{\kx_n}$ is an isomorphism. A similar
argument proves that ${\rm tr}_{\kx'_n}$ is an isomorphism for the
case of the right adjoint.
\end{proof}

Under the assumptions of the previous proposition, the same proof
also yields an equivalence
$\Phi_{\ke_n}:\Db(\kx_n)\to\Db(\kx'_n)$.

\begin{cor}\label{prop:genericalsoequiv}
Let $\ke\in\Db(\kx\times_R\kx')$, such that
$\Phi_{\ke_0}:\Db(X)\congpf\Db(X')$ is an equivalence. Then the
Fourier--Mukai transform
$\Phi_{\ke_K}:\Db(\kx_K)\congpf\Db(\kx'_K)$ is an equivalence,
where $\ke_K$ denotes the image of $\ke$ in
$\Db((\kx\times_R\kx')_K)$.
\end{cor}

\begin{proof} Indeed the inverse Fourier--Mukai functor
$\Phi_\kf:=\Phi_{\ke}^{-1}:\Db(\kx')\to\Db(\kx)$, which exists due
to  Proposition \ref{prop:extequiv}, descends to a Fourier--Mukai
transform (see Section \ref{sect:derfunc})
$$\Phi_{\kf_K}:\Db(\kx'_K)\to\Db(\kx_K),$$
which clearly is an inverse to $\Phi_{\ke_K}$.
\end{proof}

\section{Properties of the derived category of the general
fibre}\label{sect:Der}

In this section we conclude the proof of Theorem \ref{thm:main}. However, for most of the results it is enough to assume that $X$ is a smooth projective variety. More precisely, the assumption that $X$ is a surface with trivial canonical bundle is needed for the first time in Proposition \ref{prop:derallsame}.

In particular, we prove that $\Db(\kx_K)$ is indeed equivalent to
the derived category of $\coh(\kx_K)$ and we study the Serre
functor of $\Db(\kx_K)$.


\subsection{Comparing Hom-spaces}\label{subsect:Homs}

Let us now consider the pull-back under the closed embedding
$\iota_n:\kx_n\hookrightarrow \kx$ which is a right exact functor
$\iota_{n}^*:\coh(\kx)\to\coh(\kx_n)$ compatible with the
$R$-linear respectively $R_n$-linear structure of the two
categories. Its left derived functor
$$\xymatrix{L\iota_n^*:\Db(\kx)\ar[r]&\Dp(\kx_n)}$$
takes bounded complexes to perfect complexes (see Remark
\ref{rem:notperfect}). When the derived context is clear, we will
often simply write $\iota_n^*$ instead of $L\iota_n^*$. For
$E\in\Db(\kx)$ one writes
$$E_n:=\iota_n^*E=L\iota_n^*E\in \Db(\kx_n).$$ In particular, $E_0$
denotes the restriction of a complex $E$ on $\kx$ to the special
fibre $X$. Clearly, $E\in\Db(\kx)$ is trivial if $E_0\cong0$.

One needs to be careful with the pull-back under
$i_n:\kx_n\hookrightarrow\kx_{n+1}$, whose left derived functor
$i_n^*:\Dp(\kx_{n+1})\to\Dp(\kx_n)$ is well-defined for perfect
complexes but not for bounded ones (see Remark
\ref{rem:notperfect}).

\begin{lem}\label{lem:restr}
{\rm i)} For $E,E'\in\Db(\kx)$ there exists a functorial
isomorphism
$$R\Hom_{\Db(\kx)}(E,E')\otimes^{L}_RR_n\congpf
R\Hom_{\Dp(\kx_n)}(L\iota_n^*E,L\iota_n^*E').$$

\noindent{\rm ii)} For $m<n$ and $E,E'\in\Dp(\kx_n)$ there exists
a functorial isomorphism
$$R\Hom_{\Dp(\kx_n)}(E,E')\otimes^{L}_{R_n}R_m\congpf
R\Hom_{\Dp(\kx_m)}(Li_{m,n}^*E,Li_{m,n}^*E').$$\end{lem}

\begin{proof}
The proofs of i) and ii) are identical. We just consider the first
case.

Since we continue to  work under the simplifying assumption that
$\kx\to\Spf(R)$ is smooth and proper, the derived local Hom's are
functors $R\kh om_{\kx}:\Db(\kx)^{\rm
op}\times\Db(\kx)\to\Db(\kx)$.  Also, by definition $(-)\otimes
^L_RR_n$ is $L\iota_n^*:\Db(\Spf (R)) \to\Dp(\Spec (R_n))$, the
derived pull-back of the inclusion
$\iota_n:\Spec(R_n)\hookrightarrow\Spec(R)$.

Thus the assumptions of \cite[Prop.\ 7.1.2]{Ill} are satisfied and
we therefore have a functorial isomorphism
\begin{equation}
L\iota_n^*R\kh om_\kx(E,E')\congpf R\kh om_{\kx_n}
(L\iota_n^*E,L\iota_n^*E').
\end{equation}

Applying the global section functor
$R\Gamma_{\kx_n}:=R\Gamma(\kx_n,-):\Db(\kx_n)\to\Db(\Spec (R_n))$
to both sides,  one finds
$$L\iota_n^* R\Gamma_\kx R \kh om_{\kx}(E,E')\stackrel{{\rm (\ast)}}\cong R\Gamma_{\kx_n} L\iota_n^*R\kh om_\kx(E,E')\congpf
R\Gamma_{\kx_n} R\kh om_{\kx_n} (L\iota_n^*E,L\iota_n^*E').$$
Together with $R\Gamma\circ R\kh om=R\Hom$,  this proves the
assertion.

Note that in ($\ast$) we used the base change formula $L\iota_n^*
\circ R\Gamma_\kx\cong R\Gamma_{\kx_n}\circ L\iota_n^*$ which can
be easily proved by adapting the argument of \cite[Sect.\
2.4]{Ku}. More precisely, one could apply Kuznetsov's discussion
to the cartesian triangle given by $\kx_n\hookrightarrow\kx$ over
$\Spec(R_n)\hookrightarrow\Spf(R)$. Corollary 2.23 in \cite{Ku}
shows that from the flatness of $\pi:\kx\to\Spf(R)$ one cannot
only deduce the standard flat base change, but  also the above
assertion (see also \cite[Ch.\ 3, Remark 3.33]{HFM}). For flat
base change in our more general context see \cite{Spalt}.
\end{proof}

The categories $\Db(X)$ and $\Db(\kx_K)$ are $\CC$-linear
respectively $K$-linear triangulated categories with
finite-dimensional Hom-spaces. The following numerical invariants
turn out to be useful and well-behaved. For $E_0,E_0'\in\Db(X)$
one sets:
$$\chi_0(E_0,E'_0):=\sum(-1)^i\dim_\CC\Ext_X^i(E_0,E'_0)$$
and analogously for $E_K,E'_K\in\Db(\kx_K)$:
$$\chi_K(E_K,E'_K):=\sum(-1)^i\dim_K\Ext_K^i(E_K,E'_K).$$

As an immediate consequence of the discussion in Section
\ref{sect:dergen}, one finds the following two results which will be used in \cite{HMS1} to describe spherical and semi-rigid objects in $\Db(\kx_K)$, when $X$ is a smooth projective K3 surface.

\begin{cor}\label{cor:chispecgen}
For any $E,E'\in\Db(\kx)$ one has
$$\chi_0(E_0,E'_0)=\chi_K(E_K,E'_K).$$
\end{cor}

\begin{proof}
By Lemma \ref{lem:restr} i), we have an isomorphism
$R\Hom_{\Db(X)}(E_0,E'_0)\iso R\Hom_{\Db(\kx)}(E,E')\otimes^L\CC$.
Since $R$ is a DVR, we have a decomposition of the $R$-module
\[
R\Hom_{\Db(\kx)}(E,E')\iso
R\Hom_{\Db(\kx)}(E,E')_{\mathrm{free}}\oplus
R\Hom_{\Db(\kx)}(E,E')_{\mathrm{tor}}
\]
in its free and torsion part. Since for a torsion $R$-module $M$ one
has $M\otimes^L\CC=0$, this yields
$$\chi_0(E_0,E_0')=\dim_{\CC} R\Hom_{\Db(X)}(E_0,E'_0)=\dim_{\CC} (R\Hom_{\Db(\kx)}(E,E')_{\mathrm{free}}\otimes\CC).$$
On the other hand, by Proposition \ref{pro:homgamma},
$$\dim_K R\Hom_{\Db(\kx_K)}(E_K,E'_K)=\dim_K(R\Hom_{\Db(\kx)}(E,E')_{\mathrm{free}}\otimes K).$$
This concludes the proof.
\end{proof}

Of course, the single Hom-spaces could be quite different on the
special and on the general fibre, but at least the standard
semi-continuity result can be formulated in our setting.

\begin{cor}\label{cor:semicont}
Let $E,E'\in\Db(\kx)$. Then
$$\dim_\CC\Hom(E_0,E'_0)\geq\dim_K\Hom_K(E_K,E'_K).$$
\end{cor}

\begin{proof}
We know that
$$\dim_K \Hom_{\Db (\kx_K)} (E_K, E'_K) = \rk_R \Hom_{\Db (\kx)} (E,E')_{\rm free}.$$
The conclusion follows from Lemma \ref{lem:restr}.
\end{proof}

\subsection{Serre functors}\label{subsec:Serre}
The relative canonical bundles
$\omega_{\pi_n}:=\omega_{\kx_n/R_n}$ of $\pi_n:\kx_n\to\Spec(R_n)$
define a coherent sheaf $\omega_\pi$ on $\kx$, the
\emph{dualizing} or \emph{canonical line bundle}. The name is
justified by the following observation (for more general
statements see \cite{AJL,Yek}):

\begin{lem}\label{cor:AJL}
Suppose $\pi:\kx\to\Spf(R)$ is a smooth proper formal scheme of
relative dimension $d$. Then there are functorial isomorphisms
$$R\Hom_{\Db(\kx)}(E,\omega_\pi[d])\congpf R\Hom_{\Db(\Spf(R))}(R\Gamma_\kx E,R),$$
for all $E\in\Db(\kx)$.
\end{lem}

\begin{proof}
Notice that $\omega_{\pi_n}$ is the dualizing complex in
$\Dp(\kx_n)$, that is
$$R\Hom_{\Dp(\kx_n)}(E_n,\omega_{\pi_n}[d])\congpf
R\Hom_{\Dp(R_n)}(R\Gamma_{\kx_n}E_n,R_n),$$ for any
$E_n\in\Dp(\kx_n)$.

For any positive integer $n$, we have the following natural
isomorphisms (using Lemma \ref{lem:restr} twice)
\begin{equation*}
\begin{split}
R\Hom_{\Db(\kx)}(E,\omega_\pi[d])\otimes^L_R R_n & \iso R\Hom_{\Dp(\kx_n)}(L\iota_n^*E,\omega_{\pi_n}[d])\\
                                                                                  & \iso R\Hom_{\Dp(R_n)}(R\Gamma_{\kx_n}L\iota_n^*E,R_n)\\
                                                                                  & \iso R\Hom_{\Db(\Spf(R))}(R\Gamma_\kx E,R)\otimes^L_R R_n.
\end{split}
\end{equation*}
(Notice that the last isomorphism uses again $L\iota_n^*\circ
R\Gamma_\kx\cong R\Gamma_{\kx_n}\circ L\iota_n^*$ as in the proof
of Lemma \ref{lem:restr}.)
 Moreover, the resulting isomorphisms
$$f_n:\Hom_{\Db(\kx)}(E,\omega_\pi[d])\otimes^L_R R_n \congpf
R\Hom_{\Db(\Spf(R))}(R\Gamma_\kx E,R)\otimes^L_R R_n$$ are
compatible under pull-back, i.e.\ $\bar
f_{n+1}:=f_{n+1}\otimes^L_R{\rm id}_{R_n}=f_n$.

Taking the projective limits allows us to conclude the proof. More
precisely, one uses the following general argument: Suppose we are
given complexes $K^\bullet,L^\bullet\in\Db(R\text{-}\cat{Mod})$
and isomorphisms $f_n:K^\bullet\otimes^L_RR_n\congpf
L^\bullet\otimes^L_RR_n$ in $\Db(R_n\text{-}\cat{Mod})$ compatible
in the above sense. Replacing $K^\bullet$ and $L^\bullet$ by
complexes of free $R$-modules, we can assume that the $f_n$ are
morphisms of complexes. Again using the projectivity of the
modules $K^i$ and $L^i$, we deduce from the compatibility of $f_n$
and $f_{n+1}$ the existence of a homotopy $k^i:K^i\otimes R_n\to
L^{i-1}\otimes R_n$ between $f_n$ and $\bar f_{n+1}$, i.e.\
$f^i_n-\bar f^i_{n+1}=k^{i+1} d^i_K+d^{i-1}_L k^i$. Lift $k^i$ to
$h^i:K^i\otimes R_{n+1}\to L^i\otimes R_{n+1}$ and replace
$f_{n+1}$ by the homotopic one $f_{n+1}+hd_K+d_Lh$. With this new
definition one has $f_n=\bar f_{n+1}$ as morphism of complexes
homotopic to he original one. Continuing in this way, one obtains
a projective system of morphisms of complexes. The limit is then
well defined and yields an isomorphism $K^\bullet\congpf
L^\bullet$.

The functoriality of the constructions is straightforward.
\end{proof}

We are now ready to show that the derived category of the general
fibre, which is a $K$-linear category, has a Serre functor in the
usual sense. The following proposition, saying that Serre duality
holds true in $\Db(\kx_K)$, shows the advantage of working with
$\Db(\kx_K)$. The \emph{canonical bundle of the general fibre} is
by definition $\omega_{\kx_K}:=(\omega_\pi)_K\in\coh(\kx_K)$.

\begin{prop}\label{pro:serrefunctor}
Suppose $\pi:\kx\to\Spf(R)$ is a smooth proper formal scheme of
relative dimension $d$. Then the functor $E\mapsto
E\otimes\omega_{\kx_K}[d]$ is a Serre functor for the $K$-linear
category $\Db(\kx_K)$, i.e.\ there are natural isomorphisms
$$\Hom_{\Db(\kx_K)} (E_K, E'_K) \congpf \left(\Hom_{\Db(\kx_K)} (E'_K,E_K\otimes\omega_{\kx_K}[d])\right)^*,$$
for all $E_K,E'_K\in\Db(\kx_K)$, where $(-)^*$ denotes the dual
$K$-vector space.
\end{prop}

\begin{proof}
We follow the proof of \cite[Prop.\ 5.1.1]{BB}. Let $E\in\Db(\kx)$
and let $E_K\in\Db(\kx_K)$ be its image  under the natural
projection. We have
\begin{equation*}
\begin{split}
\Hom_{\Db(\kx_K)} (E_K,\omega_{\kx_K}[d]) & \iso\Hom_{\Db (\kx)} (E,\omega_\pi[d])\otimes_R K \\
                                                  & \iso\Hom_{\Db(\Spf(R))}(R\Gamma_\kx E,R)\otimes_R K \\
                                                 & \iso\Hom_R (\oplus_s R^s\Gamma_\kx E[-s], R)\otimes_R K \\
                                                 & \iso\Hom_R (R^0\Gamma_\kx E,R)\otimes_R K\\
                                                 & \iso\Hom_R (\Hom_{\Db (\kx)} (\ko_\kx,E), R)\otimes_R K \\
&\iso(\Hom_{\Db(\kx)}(\ko_\kx,E)\otimes_RK))^*\\
                                                 & \iso(\Hom_{\Db (\kx_K)} (\ko_{\kx_K},E_K))^*,
\end{split}
\end{equation*}
where the first and the last isomorphisms follow from Proposition
\ref{pro:homgamma}, while the second is Lemma \ref{cor:AJL},
and all the others are simple consequences of the fact that $R$ is
a DVR. Dualizing (with respect to $K$) we have
$$\Hom_{\Db(\kx_K)} (\ko_{\kx_K}, E_K) \iso (\Hom_{\Db(\kx_K)} (E_K, \omega_{\kx_K} [d]))^*.$$
Now, let $E,E'\in \Db (\kx)$ and let $E_K, E'_K\in\Db(\kx_K)$ be
their images. Since $E$ and $E'$ are perfect complexes, the
natural map
$$R\kh om_\kx(E',E\otimes\omega_{\pi})\to R\kh om_\kx(R \kh om_\kx(E,E'),\omega_\pi)$$
is an isomorphism. Indeed the statement is local and we can assume
$E$ and $E'$ be bounded complexes of free sheaves. In that case
the claim is obvious.

Then one concludes by
\begin{equation*}
\begin{split}
\Hom_{\Db(\kx_K)} (E_K,E'_K) & \iso \Hom_{\Db (\kx)} (E,E') \otimes_R K \\
                                                                                  & \iso \Hom_{\Db(\kx)} (\ko_\kx,R\kh om_\kx(E,E')) \otimes_R K \\
                                                                                  & \iso (\Hom_{\Db(\kx)} (R\kh om_\kx(E,E'), \omega_\pi [d]) \otimes_R K)^* \\
                                                                                  & \iso (\Hom_{\Db(\kx)} (\ko_\kx,R\kh om_\kx(R\kh om_\kx(E,E'), \omega_\pi [d])) \otimes_R K)^* \\
                                                                                  & \iso (\Hom_{\Db(\kx)} (\ko_\kx,R\kh om_\kx(E',E\otimes\omega_\pi [d])) \otimes_R K)^* \\
                                                                                  & \iso (\Hom_{\Db(\kx)} (E',E\otimes\omega_\pi [d]) \otimes_R K)^* \\
                                                                                  & \iso (\Hom_{\Db(\kx_K)} (E'_K,E_K\otimes\omega_{\kx_K} [d]))^*,
\end{split}
\end{equation*}
where the first and the last isomorphisms follow from Proposition \ref{pro:homgamma} while the third is Lemma \ref{cor:AJL}. The functoriality is clear.
\end{proof}

If $X$ is a smooth projective surface with trivial canonical bundle, then $\omega_\pi$ is trivial
and thus $\omega_{\kx_K}$ is trivial as well. Therefore, in this
case, the Serre functor of $\Db(\kx_K)$ is isomorphic to the
square of the shift functor and $\Db(\kx_K)$ is thus a K3 (or
$2$-Calabi--Yau) category as claimed in Theorem \ref{thm:main}.

\subsection{A technical result}\label{subsec:technical}
Instead of taking Verdier quotients of derived categories, one
could also consider  derived cate\-gories of Serre quotients of
the underlying abelian categories. Let us start with a few
observations that should hold for the more general situation of
the natural projection $\Db(\kb)\to\Db(\kb/\ka)$ induced by the
quotient of a (non localizing) Serre subcategory $\ka\subset\kb$
of an abelian category $\kb$. We could not find a good reference
for the general case and since the proofs are technically easier,
we restrict to the Serre subcategory $\coh(\kx)_0\subset\coh(\kx)$
with quotient $\coh(\kx_K)$.

The following technical result is probably well-know in other
contexts. It is the first step towards the proof of the
existence of an exact equivalence
$\Db(\kx_K)\cong\Db(\coh(\kx_K))$. We include the proof here for
the convenience of the reader.

\begin{prop}\label{pro:equivalence}
The natural exact functor $Q:\Db(\coh(\kx))\to\Db(\coh(\kx_K))$
induces an exact equivalence
$$\Db(\kx_K^c)=\Db(\coh(\kx))/\Db_0(\coh(\kx))\congpf\Db(\coh(\kx_K)).$$
\end{prop}

For an abelian
category $\cat{A}$, we denote by ${\rm C^b}(\cat{A})$ the abelian
category of bounded complexes of objects in $\cat{A}$ and by ${\rm
K^b}(\cat{A})$ the category of bounded complexes modulo homotopy.

\begin{lem}\label{Lemma1}
The natural projection $Q:\Db(\coh(\kx))\to\Db(\coh(\kx_K))$ is
essentially surjective.
\end{lem}

\begin{proof}
Let $F^\bullet$ be a bounded complex in the quotient category
${\rm C^b}(\coh(\kx_K))$, i.e.\ $F^i=E^i_K$ for some
$E^i\in\coh(\kx)$ and differentials $d^i\in\Hom_K(F^i,
F^{i+1})=\Hom(E^i, E^{i+1})\otimes K$ (see Proposition
\ref{prop:Homcohgen}).

Suppose $F^i=0$ for $|i|> n$ for some $n>0$. Then there exists
$N\gg0$ such that $t^Nd^i\in\Hom(E^i, E^{i+1})$. Furthermore, we
can choose $N$ large enough such that $(t^Nd^{i+1})\circ(t^Nd^i)$
is trivial in $\coh(\kx)$ for all $i$. Let $\widetilde E^\bullet$
be the complex with objects $\widetilde E^i=E^i$ and differentials
$\tilde d^i:=t^Nd^i$. Then $t^{N(n-i)}:\widetilde E^i_K\congpf
F^i$ defines an isomorphism of complexes $Q(\widetilde
E^\bullet)\congpf F^\bullet$.
\end{proof}

Thus, in particular, in order to prove Proposition
\ref{pro:equivalence}, i.e.\ that the natural functor induces an
equivalence $\Db(\kx_K^c)\congpf\Db(\coh(\kx_K)),$ it remains to
show $\Hom_{\Db(\kx_K^c)}\cong\Hom_{\Db(\coh(\kx_K))}$. By
Proposition \ref{pro:homgamma} we already know that
$\Hom_{\Db(\kx_K^c)}\cong\Hom_{\Db(\coh(\kx))}\otimes K$. Thus, we
just need to show that $\Db(\kx_K^c)\to\Db(\coh(\kx_K))$  induces
as well isomorphisms $\Hom_{\Db(\coh(\kx))}\otimes
K\cong\Hom_{\Db(\coh(\kx_K))}$. This will be the content of Lemma
\ref{Lemma4}.

In the following we shall frequently use the much easier fact that
\begin{equation}\label{eqn:complK}
\Hom_{{\rm C^b}(\coh(\kx))}(E_1^\bullet,E_2^\bullet)\otimes
K\cong\Hom_{{\rm
C^b}(\coh(\kx_K))}(Q(E_1^\bullet),Q(E_2^\bullet)),
\end{equation}
which is proved by the same argument as Proposition
\ref{prop:Homcohgen}. One only has to observe in addition that in
order to lift a morphism of complexes $f^\bullet:Q(E_1^\bullet)\to
Q(E_2^\bullet)$, one first lifts all $t^nf^i$ to $\tilde
f^i:E_1^i\to E_2^i$ for some $n\gg0$ and to make $\tilde
f^\bullet$ a map of complexes on $\kx$ one might have to
annihilate kernel and cokernel of $d_{E_2^\bullet}^i\circ \tilde
f^i-\tilde f^{i+1}\circ d_{E_1^\bullet}^i$ by multiplying with yet
another high power of $t$.

\begin{lem}\label{Lemma2}
Let $E_1^\bullet, E_2^\bullet\in{\rm C^b}(\coh (\kx))$ and let $h
\in \Hom_{{\rm C^b}(\coh (\kx))} (E_1^\bullet,E_2^\bullet)$ be
such that $Q (h)$ is a quasi-isomorphism in ${\rm C^b}(\coh
(\kx_K))$. Then there exist two complexes
$F_1^\bullet,F_2^\bullet$ and two morphisms $f_1:F_1^\bullet\to
E_1^\bullet$, $f_2:F_2^\bullet\to E_2^\bullet$ in ${\rm C^b}(\coh
(\kx))$ such that $Q (f_1)$ and $Q(f_2)$ are isomorphisms in ${\rm
C^b}(\coh (\kx_K))$ and $Q(f_2)^{-1}\circ Q(h)\circ
Q(f_1)=Q(\gamma)$, with $\gamma$ a quasi-isomorphism in ${\rm
C^b}(\coh(\kx))$.
\end{lem}

\begin{proof}
The proof is based on calculations similar to the ones in the
proof of Lemma \ref{Lemma1}, we will therefore be brief. We shall
outline a construction that yields a $\gamma$ inducing an
isomorphism in the lowest cohomology and leave the higher
cohomologies to the reader.

Up to shift, we can assume that $E_1^\bullet$, $E_2^\bullet$ and
hence $h$
are concentrated in $[0,r]$. Since $Q(h)$ is a quasi-isomorphism
in ${\rm C^b}(\coh (\kx_K))$, the induced maps
$\kh^i(h):\kh^i(E_1^\bullet)\to\kh^i(E_2^\bullet)$ on cohomology
have kernels and cokernels in $\coh (\kx)_0$.

In the following discussion we use the observation that for any
$K\in\coh(\kx)$ and $n\gg 0$, the sheaf $t^nK$ is $R$-flat and the
cokernel of the inclusion $t^nK\hookrightarrow K$ is an object of
$\coh(\kx)_0$ (see Section \ref{sect:abgen}).

\vspace{0.2cm}

We first construct a complex $Z_{1,0}^\bullet\in{\rm
C^b}(\coh(\kx))$ and a morphism $f'_{1,0}:Z_{1,0}^\bullet\to
E_1^\bullet$ such that $\ker(\kh^0 (h\circ f'_{1,0}))$ is trivial.
If $n\gg 0$, then $Z^0_{1,0}:=t^n E_1^0$ is $R$-flat and the
inclusion $i'_{1,0}:Z^0_{1,0}:=t^n E_1^0\hookrightarrow E_1^0$ is
an isomorphism in $\coh(\kx_K)$. Then the map of complexes
$$\xymatrix{Z^0_{1,0}\ar[d]_{f'_{1,0}}:&0\ar[rr]^{} \ar[d]^{} && Z^0_{1,0} \ar[rr]^{d_{E_1}^0 \circ i'_{1,0}} \ar[d]^{i'_{1,0}} && E_1^1 \ar[rr]^{d_{E_1}^1} \ar[d]^{\id} && E_1^2 \ar[r]^{} \ar[d]^{\id} & \ldots  \\
             E_1^\bullet:&    0 \ar[rr]^{} && E_1^0 \ar[rr]^{d_{E_1}^0} && E_1^1 \ar[rr]^{d_{E_1}^1} && E_1^2 \ar[r]^{} &
                 \ldots
                 }$$
yields an isomorphism in ${\rm C^b}(\coh (\kx_K))$. As a subsheaf
of the $R$-flat sheaf $Z^0_{1,0}$, the kernel $\ker(d_{E_1}^0
\circ i'_{1,0})$ is also $R$-flat. Since $\ker(\kh^0(h\circ
f'_{1,0}))\mono\ker(d_{E_1}^0\circ i'_{1,0})$ and
$\ker(\kh^0(h\circ f'_{1,0}))\in\coh(\kx)_0$, this implies
$\ker(\kh^0(h\circ f'_{1,0})=0$. To simplify the notation, we
assume henceforth $E_1^\bullet=Z_{1,0}^\bullet$ and $h=h\circ
f'_{1,0}$, i.e.\ that $\kh^0(h)$ is injective.

\vspace{0.2cm}

Now we define two complexes
$F_{1,0}^\bullet,F_{2,0}^\bullet\in{\rm C^b}(\coh(\kx))$ and
morphisms $f_{1,0}:F_{1,0}^\bullet\to E_1^\bullet$,
$f_{2,0}:F_{2,0}^\bullet\to E_2^\bullet$ yielding isomorphisms in
${\rm C^b}(\coh(\kx_K))$, such that there exists a morphism
$h_0:F_{1,0}^\bullet\to F_{2,0}^\bullet$ with $h\circ
f_{1,0}=f_{2,0}\circ h_0$ and
$\ker(\kh^0(h_0))=\coker(\kh^0(h_0))=0$.

To this end, consider the diagram
\begin{equation}\label{4}
\xymatrix{& & 0 \ar[d] & & \\
            & 0 \ar[d] \ar[r] & A^0 \ar[d] \ar[r] &  Q^0 \ar[d]^{\id} &  \\
                 0 \ar[r]  & \ker(d_{E_1}^0) \ar[d] \ar[r] & E_1^0 \ar[d]^{h^0} \ar[r] & Q^0 \ar[r] & 0 \\
                 & E_2^0 \ar[d] \ar[r]^{\id} & E_2^0 \ar[d] & & \\
                 & C^0 \ar[d] \ar[r] & B^0 \ar[r] \ar[d] & 0  \\
                 & 0 & 0 & &
                 }
\end{equation}
with exact rows and columns. The aim is to reduce to the case
where $C^0$ is $R$-flat.

Let $D^0$ denote the cokernel of $A^0\to Q^0$. Choose $n\gg 0$ and
consider the short exact sequence $\xymatrix{0 \ar[r]& D_{\rm
flat}^0 := t^n D^0 \ar[r]& D^0 \ar[r]& D_{\rm tor}^0 \ar[r]& 0,}$
where $D_{\rm tor}^0 \in \coh (\kx)_0$ and $D_{\rm flat}^0$ is
$R$-flat, and define $F^0_{1,0}$ as the kernel of the composition
$E_1^0\to Q^0\to D_{\rm tor}^0$.
By construction, the map of complexes
$$\xymatrix{F_{1,0}^\bullet\ar[d]_{f_{1,0}}:&0 \ar[rr]^{} \ar[d]^{} && F^0_{1,0} \ar[rr]^{d_{E_1}^0 \circ i_{1,0}} \ar[d]^{i_{1,0}} && E_1^1 \ar[rr]^{d_{E_1}^1} \ar[d]^{\id} && E_1^2 \ar[r]^{} \ar[d]^{\id} & \ldots  \\
                 E_1^\bullet:&0 \ar[rr]^{} && E_1^0 \ar[rr]^{d_{E_1}^0} && E_1^1 \ar[rr]^{d_{E_1}^1} && E_1^2 \ar[r]^{} & \ldots
                }$$
yields an isomorphism in ${\rm C^b} (\coh (\kx_K))$. Note that by
construction the inclusion $A^0 \hookrightarrow E_1^0$ factorizes
through $F^0_{1,0}$  and that  the inclusion $\ker(d_{E_1}^0)
\hookrightarrow E_1^0$ factorizes through $F^0_{1,0}$. Replace
$E_1^\bullet$ by $F_{1,0}^\bullet$ and $h$ by $h\circ f_{1,0}$.
Now, in  the corresponding diagram \eqref{4} the inclusion $A^0
\hookrightarrow Q^0$ has an $R$-flat cokernel.

Next, consider the exact sequence $\xymatrix{ 0\ar[r]& B_{\rm
flat}^0 \ar[r]& B^0 \ar[r]& B_{\rm tor}^0 \ar[r]& 0}$ and define
$F^0_{2,0}$ as the kernel of the composition $E_2^0\to B^0\to
B_{\rm tor}^0$, which naturally contains $\im(h^0)$.
As before, the map of complexes
$$\xymatrix{F_{2,0}^\bullet\ar[d]_{f_{2,0}}:0 \ar[rr]^{} \ar[d]^{} && F^0_{2,0} \ar[rr]^{d_{E_2}^0 \circ j_{2,0}} \ar[d]^{j_{2,0}} && E_2^1 \ar[rr]^{d_{E_2}^1} \ar[d]^{\id} && E_2^2 \ar[r]^{} \ar[d]^{\id} & \ldots  \\
                 E_2^\bullet:0 \ar[rr]^{} && E_2^0 \ar[rr]^{d_{E_2}^0} && E_2^1 \ar[rr]^{d_{E_2}^1} && E_2^2 \ar[r]^{} & \ldots
                }$$
yields an isomorphism in ${\rm C^b}(\coh(\kx_K))$ and $h$
factorizes through $f_{2,0}$. Replace $E_2^\bullet$ by
$F_{2,0}^\bullet$ and consider the corresponding diagram
\eqref{4}. Observe that now $C^0$ is $R$-flat (use the Snake
Lemma). Since $\coker( \kh^0 (h))$ injects into $C^0$ and belongs
to $\coh(\kx)_0$, it must be trivial, as wanted.
\end{proof}

In the spirit of Proposition \ref{prop:Homcohgen} one can describe
the homomorphisms in the derived category of the quotient as
follows.

\begin{lem}\label{Lemma4}
For all complexes $E_1^\bullet,E_2^\bullet\in\Db(\coh(\kx))$ the
natural exact functor $Q$ induces isomorphisms
$$Q\otimes K:\Hom_{\Db(\coh(\kx))}(E_1^\bullet,E_2^\bullet)
\otimes_RK\congpf\Hom_{\Db(\coh(\kx_K))}(Q(E_1^\bullet),Q(E_2^\bullet)).$$
\end{lem}

\begin{proof}
We will prove the bijectivity of $Q\otimes K$ in two steps.

\vspace{0.2cm}

\noindent i) Injectivity. Let $f \in \Hom_{\Db (\coh (\kx))}
(E_1^\bullet,E_2^\bullet)$ such that $Q (f) =0$. By definition,
$f$ may be represented by $$\xymatrix{E_1^\bullet&\ar[l]_-{s_0}
F_0^\bullet \ar[r]^-{g}& E_2^\bullet,}$$ with $s_0$ a
quasi-isomorphism in ${\rm C^b} (\coh (\kx))$. Since $Q(f) = 0$,
there exists a commutative diagram in ${\rm K^b} (\coh (\kx_K))$
of the form
\begin{equation*}
\xymatrix{& Q(F_0^\bullet) \ar[dl]_{Q (s_0)} \ar[dr]^{Q(g)} &\\
                 Q(E_1^\bullet) && Q(E_2^\bullet),\\
                 & \widetilde{F}_1^\bullet \ar[ul]^{\tilde s_1} \ar[ur]_{0} \ar[uu]^{\tilde{h}} &
                }
\end{equation*}
with $\widetilde{s_1}$ and $\widetilde{h}$ quasi-isomorphisms in
${\rm C^b} (\coh (\kx_K))$. By Lemma \ref{Lemma1} and
(\ref{eqn:complK}), we can assume that
$\widetilde{s_1},\widetilde{h}$, and $\widetilde F_1^\bullet$ are
in the image of $Q$, i.e.\ $\tilde s_1=Q(s_1)$, $\tilde h=Q(h)$,
and $\widetilde{F}_1^\bullet = Q (F_1^\bullet)$. By Lemma
\ref{Lemma2} we have a commutative diagram in ${\rm K^b}(\coh
(\kx_K))$
\begin{equation*}
\xymatrix{& Q(F_3^\bullet) &\\
                 & Q(F_0^\bullet) \ar[u]^{Q(f_2)^{-1}} \ar[dl]_{Q (s_0)} \ar[dr]^{Q(g)} &\\
                 Q(E_1^\bullet) && Q(E_2^\bullet),\\
                 & Q(F_1^\bullet) \ar[ul]^{Q(s_1)} \ar[ur]_{0} \ar[uu]^{Q(h)} &\\
                 & Q(F_2^\bullet) \ar[u]^{Q(f_1)} &
                }
\end{equation*}
with $\gamma$ a quasi-isomorphism in ${\rm C^b}(\coh(\kx))$ such
that $Q(\gamma)=Q(f_2)^{-1}\circ Q(h)\circ Q(f_1)$. So we have a
commutative diagram in ${\rm K^b}(\coh(\kx_K))$
\begin{equation*}
\xymatrix{& Q(F_3^\bullet) \ar[dr]^{Q(g \circ f_2)} &\\
                 Q(F_2^\bullet) \ar[rr]^{0} \ar[ur]^{Q(\gamma)} && Q(E_2^\bullet).
                }
\end{equation*}
Hence, one finds  $\widetilde{k}^i:Q(F_2^i) \mor Q(E_2^{i-1})$ in
$\coh(\kx_K)$ such that
$$d_{Q(E_2)}\circ\widetilde{k}+\widetilde{k}\circ d_{Q(F_2)}-Q(g\circ f_2\circ\gamma)=0$$
in $\coh (\kx_K)$. By Proposition \ref{prop:Homcohgen}, there
exists $N\gg 0$ such that $t^N \widetilde{k}=Q(k)$ and
$$d_{E_2}\circ k+k\circ d_{F_2}-((t^N(g\circ f_2))\circ\gamma)=0$$
in $\coh(\kx)$. So $(t^N (g \circ f_2)) \circ \gamma = 0$ in ${\rm
K^b} (\coh (\kx))$. Therefore, there is a quasi-isomorphism
$\gamma':E_2^\bullet\mor F_4^\bullet$ in ${\rm K^b} (\coh (\kx))$
such that $\gamma' \circ (t^N (g \circ f_2)) = 0$.

Then by Lemma \ref{Lemma2}, saying in particular that $Q(f_2)$ is
an isomorphism in ${\rm C^b}(\coh(\kx_K))$, and by
(\ref{eqn:complK}), there exist $h\in\Hom_{{\rm
C^b}(\coh(\kx))}(F_0^\bullet,F_3^\bullet)$ and $n\gg 0$, such that
$t^n f_2\circ h={\rm id}$ and hence
$\gamma'\circ(t^{n+N}g)=\gamma'\circ(t^N(g\circ f_2))\circ(t^n h)
= 0$ in ${\rm K^b} (\coh (\kx))$. Hence $\gamma' \circ (t^{n+N} f)
= 0$ in $\Db (\coh (\kx))$. Since $\gamma'$ is a
quasi-isomorphism, this yields $t^{n+N} f =0$ in $\Db (\coh
(\kx))$.

\vspace{0.2cm}

\noindent ii) Surjectivity. Let
$\widetilde{f}\in\Hom_{\Db(\kx_K)}(Q(E_1^\bullet),Q(E_2^\bullet))$.
Again by Lemma \ref{Lemma1} and (\ref{eqn:complK}) we can assume
that $\widetilde{f}$ is of the form
$$\xymatrix{Q (E_1^\bullet)&&\ar[ll]_-{Q(s_0)}Q(F_0^\bullet)\ar[rr]^-{Q(g)}&&Q(E_2^\bullet).}$$
Applying Lemma \ref{Lemma2} to $Q(s_0)$ we get a commutative
diagram in ${\rm C^b}(\coh (\kx_K))$:
\begin{equation}\label{eqn:compos}
\xymatrix{&& Q(F_2^\bullet)\ar[dl]_{Q(s_0\circ f_1)}\ar[dr]^{Q(f_1)} &&\\
                 & Q(E_1^\bullet)\ar[dl]_{Q(f_2)^{-1}}\ar[dr]^{\id} && Q(F_0^\bullet)\ar[dl]_{Q(s_0)}\ar[dr]^{Q(g)}  &\\
                 Q(F_1^\bullet) && Q(E_1^\bullet) && Q(E_2^\bullet),
                }
\end{equation}
with $Q(f_2)^{-1}\circ Q(s_0\circ f_1)=Q(\gamma)$ and $\gamma$ a
quasi-isomorphism, giving rise to a morphism
$\alpha\in\Hom_{\Db(\coh(\kx))}(F_1^\bullet,E_2^\bullet)$ such
that $Q(\alpha)$ is represented by (\ref{eqn:compos}).

If
$\widetilde{\beta}\in\Hom_{\Db(\coh(\kx_K))}(Q(F_1^\bullet),Q(E_1^\bullet))$
corresponds to the diagram
$$\xymatrix{Q (F_1^\bullet)&&\ar[ll]_-{Q(f_2)^{-1}}Q(E_1^\bullet)\ar[rr]^-{\id}&&Q(E_1^\bullet),}$$
we have $\widetilde{f}\circ\widetilde{\beta}=Q(\alpha)$ and Hence
$\widetilde{f}=Q(\alpha)\circ\widetilde{\beta}^{-1}$. Applying
(\ref{eqn:complK}) to $Q(f_2)^{-1}$ one finds $n\gg0$ and $g$,
such that $t^n\widetilde{\beta}^{-1}=Q(g_2)$. Thus
$t^n\widetilde{f}=Q(\alpha)\circ Q(g_2)$, as desired.
\end{proof}

\subsection{Back to the general fibre}\label{subsec:genfib}
In the definitions of $\Db(\kx_K)$ and $\Db(\kx^c_K)$ one divides by the
categories $\Db_0(\kx)$ and $\Db(\coh(\kx)_0)$ which, by Proposition \ref{prop_dertorsion}, are equivalent. The
categories $\Db_{\rm coh}(\Mod\kx)$ and $\Db(\coh(\kx))$ are in
general not equivalent, so neither should be their quotients
$\Db(\kx_K^c)$ and $\Db(\kx_K)$. However, for surfaces with
trivial canonical bundle the situation is slightly better.

In the sequel we will write, by abuse of notation, $Q(E)=E_K$ where $$Q:\Db(\coh(\kx))\to\Db(\coh(\kx_K))$$ is defined as in Lemma \ref{Lemma1}.

\begin{prop}\label{prop:derallsame}
Suppose $\kx\to\Spf(R)$ is a smooth proper formal scheme of
dimension two with trivial canonical bundle, i.e.\ $\omega_\pi\iso
\ko_\kx$. Then the natural exact functor
$$\Db(\coh(\kx))\to\Db(\kx)\to\Db(\kx_K)$$
induces an exact equivalence
$$\xymatrix@R=1pt{\Db(\coh(\kx_K))\ar[r]^-\sim&\Db(\coh(\kx))/\Db_0(\coh(\kx))\ar[r]^-\sim&
\Db(\kx)/\Db_0(\kx)\\
&=\Db(\kx^c_K)&=\Db(\kx_K).}$$
\end{prop}

\begin{proof}
The first equivalence is the content of Proposition
\ref{pro:equivalence}, so only the second equi\-valence needs a
proof. By the universal property of localization and Remark
\ref{rem:smallest}, the induced functor
$\Db(\kx_K^c)\to\Db(\kx_K)$ exists. We need to prove it to be an
equivalence.

Let us first show  that it is fully faithful.
Using induction on cohomologies, this would follow from
\begin{equation}\label{eqn:toprove}
\Hom_{\Db (\kx_K^c)} (E_K, F_K[i]) \congpf \Hom_{\Db(\kx_K)}
(E_K,F_K[i])
\end{equation}
for all objects $E_K, F_K\in\coh(\kx_K)$ and all $i\in\NN$. Here we use that the natural
$K$-linear functor $\coh(\kx_K)\to\Db(\kx_K)$, which by
Propositions \ref{prop:Homcohgen} and \ref{pro:homgamma} is fully
faithful, identifies $\coh(\kx_K)$ with the heart of a bounded
$t$-structure on $\Db(\kx_K)$ (see, e.g.\ \cite[Lemma 3.2]{B2}).

In order to prove (\ref{eqn:toprove}), we imitate the proof of
\cite[Prop.\ 5.2.1]{BB}. For fixed $F\in\coh(\kx)$, write
$\Ext_I^*(-,F)$ and $\Ext^*_{I\! I}(-,F)$ for the two contravariant
$\delta$-functors $\Ext^*_{\Db (\coh (\kx))}(-,F) \otimes_R K$ and
$\Ext^*_{\Db (\kx)}(-,F) \otimes_R K$ on $\coh(\kx)$ with values
in the category of $K$-vector spaces. They coincide in degree zero
and $\Ext^*_I(-,F)$ is clearly universal. Thus, it suffices to
prove that also $\Ext^*_{I\! I}(-,F)$ is universal. By
Grothendieck's result (see \cite[Thm.\ 1.3.A]{HartAG}), this would
follow from $\Ext^i_{I\! I}(-,F)$ being coeffaceable for $i>0$.
Recall that $\Ext^i_{I\! I}(-,F)$ is coeffaceable if for any $E \in
\coh (\kx)$, there exists an epimorphism $E' \epi E$ in $\coh
(\kx)$ such that the induced map $\Ext^i_{I\! I} (E, F) \to
\Ext_{I\! I}^i (E',F)$ is zero. Clearly $\Ext^*_I(-,F)$ is universal
and $\Ext^1_I (E,F)\iso\Ext^1_{I\! I}(E,F)$ (use that
$\coh(\kx_K)$ is the heart of a bounded $t$-structure on
$\Db(\kx_K)$ and so the extensions in the abelian category coincides with those in the triangulated category). An easy modification of Grothendieck's original
argument shows that it is enough to prove that $\Ext^i_{I\!
I}(-,F)$ is coeffaceable for $i>1$. Moreover, by Proposition
\ref{pro:serrefunctor}, we also have $\Ext_{I\! I}^i(E,F)=0$ for
$i > 2$. Hence we only have to show that $\Ext_{I\! I}^2(-,F)$ is
effaceable.

By Lemma \ref{lem:accessory7}, for all rational sections $s$ of
$\kx$ over $R$, there exists a positive integer $n$ such that
$\Ext_{I\! I}^2(\km_s^n E,F)=0$, where $\km_s$ denotes the ideal
sheaf corresponding to $s$. Then, take $s$ and $s'$ two disjoint rational
$R$-sections of $\kx$ and choose $n$ such that $\Ext_{I\!
I}^2(\km_s^n E,F)=\Ext_{I\! I}^2(\km_{s'}^n E, F)=0$. Since the
canonical map $\km_s^n E\oplus\km_{s'}^n E\to E$ is surjective, we
conclude by setting $E':=\km_s^n E\oplus\km_{s'}^n E$.

Finally, one shows that $\Db(\kx_K^c)\to\Db(\kx_K)$ is also
essentially surjective. Indeed, since  $\coh(\kx_K)$ is in the
natural way a heart of $t$-structures on both categories, this
follows by induction over the length of complexes and the full
faithfulness proved before.
\end{proof}

\begin{lem}\label{lem:accessory7}
Let $E, F \in \coh (\kx)$ and let $s$ be a rational section of
$\kx$ over $R$ whose ideal sheaf in $\ko_\kx$ is $\km_s$ . Then
there exists a positive integer $n$ such that $$\Ext_{\Db (\kx)}^2
(\km_s^n E, F) \otimes_R K = 0.$$
\end{lem}

\begin{proof}
By Proposition \ref{pro:serrefunctor} it suffices to show that,
for $n\gg 0$ one has $\Hom_{\Db (\kx)} (F, \km_s^n E) \otimes_R K
= 0$. Since $\Hom_{\Db (\kx)} (F, \km_s^n E) \otimes_R K$ is
finite dimensional over $K$, it is sufficient to show that, for $a
\in \NN$, there exists $b > a$ such that the natural inclusion
$\Hom_{\Db (\kx)} (F, \km_s^b E) \otimes_R K \subset\Hom_{\Db
(\kx)} (F, \km_s^a E) \otimes_R K$ is strict. (Use that
$(\km_s^bE)_K\to(\km_s^aE)_K$ is still injective.)

Pick a non-zero $f \in \Hom_{\Db (\kx)} (F, \km_s^a E) \otimes_R
K$. After multiplying with some power of $t$, we can assume $f \in
\Hom_{\Db (\kx)} (F, \km_s^a E)$. Consider the exact sequence
$$0 \to E_{\rm tor} \to E \to E_{\rm f} \to 0,$$
with $E_{\rm tor} \in \coh (\kx)_0$ and $E_{\rm f}$ flat over $R$.
Consider the induced map $f : F \to \km_s^a E_{\rm f}$. This is
non-zero, since $f$ is not a torsion element. It is sufficient to
show that there exists an integer $b > a$ such that $t^n f \notin
\Hom_{\Db (\kx)} (F, \km_s^b E_{\rm f})$, for all $n \in \NN$.
Thus, if $G:={\rm Im}(f)$, it is enough to show that
$G_K\subset(\km_s^aE_{\rm f})_K$ is not contained in
$\bigcap_k(\km_s^kE_{\rm f})_K$.

Suppose to the contrary that $G_K\subset \bigcap (\km_s^kE_{\rm
f})_K$. We will show that then $G\subset\bigcap\km_s^kE_{\rm f}$,
which by the Krull Intersection Theorem would show $G=0$. Indeed,
if $G\subset\km_s^kE_{\rm f}$, then also $G\subset
\km_s^{k+1}E_{\rm f}$, as the induced map
$G\to(\km_s^k/\km_s^{k+1})E_{\rm f}$ becomes the trivial map in
$\coh(\kx_K)$, but $E_{\rm f}$ is $R$-flat and $t\notin\km_s$.
\end{proof}

This applies to the case of $X$ a smooth complex projective
surface with trivial canonical bundle and proves the last part of
Theorem \ref{thm:main}.

\begin{remark}
It should be possible to deduce from Proposition
\ref{prop:derallsame} that for smooth formal surfaces with trivial
canonical bundle in fact $\Db(\coh(\kx))\congpf \Db_{\rm
coh}(\Mod{\kx})\cong\Db(\kx)$ is an equivalence, but we shall not
use this.
\end{remark}



\appendix

\section{Verdier quotients and Serre subcategories}\label{sec:Appendix}

This appendix collects known facts and definitions concerning quotients by Serre subcategories and Verdier quotients which were used throughout this paper. The main source we follow in the exposition is \cite{Neeman}. The reader is encouraged to look at Chapter $2$ and Appendix A of \cite{Neeman} for a complete and readable account. Notice that we forget all set-theoretical issues which, in the case considered in the paper, do not give rise to problems (see \cite[Sect.\ 2.2]{Neeman}).

\medskip

\subsection{Verdier quotients}\label{subsec:verdier}

Let $\cat{T}$ be a triangulated category with shift functor $\Sigma:\cat{T}\to\cat{T}$ (see \cite[Ch.\ 1]{Neeman}). A full additive subcategory $\cat{D}$ is a \emph{triangulated subcategory} if every object isomorphic to an object in $\cat{D}$ is in $\cat{D}$ and the inclusion functor $i:\cat{D}\hookrightarrow\cat{T}$ is a triangulated functor with the additional requirement that, for any $D\in\cat{D}$ the isomorphism $\phi_D:i(\Sigma(D))\to\Sigma(i(D))$ is the identity.

\begin{definition}\label{def:thick}
	A subcategory $\cat{D}$ of $\cat{T}$ is \emph{thick} if it is triangulated and contains all direct summands of its objects.
\end{definition}

If $\cat{D}$ is a triangulated subcategory of $\cat{T}$ one can form the \emph{Verdier quotient} $\cat{T}/\cat{D}$ which is a triangulated category whose objects are the same as those of $\cat{T}$. To define the morphisms in $\cat{T}/\cat{D}$ first consider the collection $\cat{Mor}_{\cat{D}}$ of morphism $f:T_1\to T_2$ in $\cat{T}$ sitting in an exact triangle
\[
\xymatrix{T_1\ar[r]^{f}& T_2\ar[r]& Z\ar[r]&\Sigma(T_1)}
\]
with $Z\in\cat{D}$. A morphism in $\cat{T}/\cat{D}$ between $T_1$ and $T_2$ is an equivalence class of diagrams $$(\xymatrix{T_1&
\ar[l]_-{f} T_0\ar[r]^-{g}& T_2})$$ with $f\in\cat{Mor}_{\cat{D}}$. We say that $(\xymatrix{T_1&
\ar[l]_-{f_1} T_0\ar[r]^-{g_1}& T_2})$ and $(\xymatrix{T_1&
\ar[l]_-{f_2} S_0\ar[r]^-{g_2}& T_2})$ are equivalent if there is a third diagram  $(\xymatrix{T_1&
\ar[l]_-{f_3} Z_0\ar[r]^-{g_3}& T_2})$ and morphisms $u:Z_0\to T_0$ and $v:Z_0\to S_0$ in $\cat{Mor}_{\cat{D}}$ making the following diagram commutative
\[
\xymatrix{& T_0 \ar[dl]_{f_1} \ar[dr]^{g_1} &\\
                 T_1 &Z_0\ar[l]_{f_3}\ar[r]^{g_3}\ar[u]^{u}\ar[d]^{v}& T_2.\\
                 &  S_0\ar[ul]^{f_2}\ar[ur]_{g_2} &
                }
\]
Roughly speaking, all morphisms in $\cat{Mor}_{\cat{D}}$ become invertible.

Let $Q:\cat{T}\to\cat{T}/\cat{D}$ be the natural triangulated functor which is called the \emph{Verdier localization}. The kernel of $Q$ (i.e.\ the full additive subcategory of $\cat{T}$ consisting of objects mapped to zero by $Q$) is thick (see \cite[Rmk.\ 2.1.7]{Neeman}). Hence, if $\cat{D}$ is thick, the kernel of $Q$ coincides with $\cat{D}$.

\subsection{Serre subcategories}\label{subsec:ASerre}

Let $\cat{A}$ be an abelian category and let $\cat{B}\subseteq\cat{A}$ be a full abelian subcategory. We say that $\cat{B}$ is \emph{thick} if for $B_1,B_2\in\cat{B}$ and any short exact sequence
\[
0\to B_1\to A\to B_2\to 0
\]
in $\cat{A}$, then $A$ belongs to $\cat{B}$ as well.

\begin{definition}\label{def:serre}
	A thick full subcategory $\cat{B}$ is a \emph{Serre subcategory} of $\cat{A}$ if
	
	i) Every object of $\cat{A}$ isomorphic to an object of $\cat{B}$ is in $\cat{B}$;
	
	ii) Every quotient or subobject in $\cat{A}$ of an object in $\cat{B}$ is in $\cat{B}$.
\end{definition}

Given a Serre subcategory $\cat{B}$ of an abelian category $\cat{A}$ one can construct the quotient $\cat{A}/\cat{B}$ where the objects of $\cat{A}/\cat{B}$ are the same as those of $\cat{A}$. On the other hand, a morphism $A_1\to A_2$ in $\cat{A}/\cat{B}$ is an equivalence class of diagrams $$(\xymatrix{A_1&
\ar[l]_-{s} A_0\ar[r]^-{t}& A_2})$$ with $\ker(s),\COKE(s)\in\cat{B}$. The equivalence relation we mentioned has a definition which is analogue to the one explained in Section \ref{subsec:verdier} (see \cite[Sect.\ A.2]{Neeman}).

A key fact is the following (see \cite[Lemma A.2.3]{Neeman}):

\begin{lem}\label{lem:serre}
	The category $\cat{A}/\cat{B}$ is abelian. The natural functor $Q:\cat{A}\to\cat{A}/\cat{B}$ is exact and takes object of $\cat{B}$ to objects in $\cat{A}/\cat{B}$ isomorphic to zero. Furthermore, $Q$ is universal with this property. The subcategory $\cat{B}\subseteq\cat{A}$ is the full subcategory of all objects $B\in\cat{A}$ such that $Q(A)$ is isomorphic to zero.
\end{lem}

\bigskip

{\small\noindent {\bf Acknowledgements.} We thank the referee for many suggestions that improved the exposition of the paper. We gratefully acknowledge
the support of the following institutions: Hausdorff Center for
Mathematics, IHES, Imperial College, Istituto Nazionale di Alta
Matematica, Max--Planck Institute, and SFB/TR 45.}



\begin{thebibliography}{99}


\bibitem{AJL} L.\ Alonso-Tarr\'{\i}o, A.\ Jerem\'{\i}as-L\'opez, J.\ Lipman,
\emph{Duality and Flat Base Change on Formal Schemes}, Contemp.\
Math.\ {\bf 244} (1999).


\bibitem{Berthelot} P.\ Berthelot,
\emph{Cohomologie rigide et cohomologie rigide \`a support propre
I},  Pr\'epublication IRMAR 96-03, (1996).

\bibitem{BB} A. Bondal, M. Van den Bergh, \emph{Generators and
representability of functors in commutative and noncommutative
geometry}, Moscow Math.\ J.\ {\bf 3} (2003), 1--36.

\bibitem{B2} T.\ Bridgeland,
\emph{Stability conditions on triangulated categories}, Ann.
Math.\ {\bf 166} (2007), 317--346.
%

\bibitem{Cal1} A.\ C\u{a}ld\u{a}raru,
\emph{The Mukai pairing, I: The Hochschild structure},
math.AG/0308079.


\bibitem{HartAG} R.\ Hartshorne,
\emph{Algebraic Geometry}, GTM 52, Springer (1977).

\bibitem{HFM} D.\ Huybrechts,
\emph{Fourier--Mukai transforms in algebraic geometry}, Oxford
Mathematical Monographs (2006).

\bibitem{HMS} D.\ Huybrechts, E.\ Macr\`i, P.\ Stellari,
\emph{Stability conditions for generic K3 surfaces}, Compositio
Math. {\bf 144} (2008), 134--162.

\bibitem{HMS1} D.\ Huybrechts, E.\ Macr\`i, P.\ Stellari,
\emph{Derived equivalences of K3 surfaces and orientation}, to appear in: Duke Math.\ J.\,
arXiv:0710.1645.

\bibitem{Ill} L.\ Illusie,
\emph{G\'en\'eralit\'es sur les conditions de finitude dans les
cat\'egories d\'eriv\'ees}, in  SGA 6, Lecture Notes Math.\ {\bf
225}, Springer (1971).

\bibitem{IllFGA} L.\ Illusie,
\emph{Grothendieck's existence theorem in formal geometry}, in
Fundamenal Algebraic Geometry, Grothendieck's FGA explained. ed.
B.\ Fantechi et al. Math.\ Surveys and Mon.\ {\bf 123}, AMS
(2005).

\bibitem{Ku} A.\ Kuznetsov,
\emph{Hyperplane sections and derived categories}, Izv.\ Math.\
{\bf 70} (2006), 447--547.

\bibitem{Neeman} A.\ Neeman,
\emph{Triangulated categories}, Annals  Math.\ Stud.\ {\bf 148}
(2001).

\bibitem{Or1} D.\ Orlov,
\emph{Equivalences of derived categories and K3 surfaces}, J.\
Math.\ Sci.\ {\bf 84} (1997), 1361--1381.

\bibitem{Ray} M.\ Raynaud,
\emph{G\'eom\'etrie analytique rigide d'apr\`es Tate, Kiehl},
M\'em.\ SMF {\bf 39-40} (1974), 319--327.

\bibitem{RZ} M.\ Rapoport, T.\ Zink,
\emph{Period spaces for $p$-divisible groups}, Annals Math.\
Stud.\ {\bf 141} (1996).


\bibitem{Spalt} N. Spaltenstein, \emph{Resolutions of unbounded complexes}, Compositio Math. {\bf 65} (1988), 121--154.

\bibitem{Sz} B.\ Szendr\H{o}i,
\emph{Diffeomorphisms and families of Fourier--Mukai transforms in
mirror symmetry}, Applications of Alg.\ Geom.\ to Coding Theory,
Phys.\ and Comp.\ NATO Science Series. Kluwer (2001), 317--337.


\bibitem{Yek} A.\ Yekutieli,
\emph{Smooth formal embedings and the residue complex}, Canadian
J.\ Math.\ {\bf 50} (1998), 863--896.

\end{thebibliography}
\end{document}